\documentclass[12pt]{article}
\usepackage{amsfonts,amssymb,latexsym,amsmath}

\newtheorem{theorem}{Theorem}[section]
\newtheorem{lemma}[theorem]{Lemma}
\newtheorem{corollary}[theorem]{Corollary}
\newtheorem{proposition}[theorem]{Proposition}

\newenvironment{proof}
{\par\addvspace{0.3cm}\noindent{\rm Proof. }}
{\nopagebreak\mbox{}\hfill $\Box$\par\addvspace{0.25cm}}

\renewcommand{\Re}{\mbox{\rm Re\,}}
\renewcommand{\Im}{\mbox{\rm Im\,}}

\newcommand{\C}{{\mathbb C}}
\newcommand{\Z}{{\mathbb Z}}
\newcommand{\T}{{\mathbb T}}
\renewcommand{\kappa}{\varkappa}
\newcommand{\qed}{\hfill $\Box$}
\newcommand{\be}{\begin{equation}}
\newcommand{\ee}{\end{equation}}

\newcommand{\bq}{\begin{eqnarray}}
\newcommand{\eq}{\end{eqnarray}}
\newcommand{\nn}{\nonumber}
\newcommand{\ba}{\begin{array}}
\newcommand{\ea}{\end{array}}
\newcommand{\bt}{\bar{\tau}}
\newcommand{\wt}[1]{\widetilde{#1}}
\newcommand{\iv}{^{-1}}
\newcommand{\iy}{\infty}
\newcommand{\al}{\alpha}

\newcommand{\cL}{{\cal L}}

\newcommand{\cA}{\mathcal{A}}

\newcommand{\cC}{\mathcal{C}}

\newcommand{\ta}{\tilde{a}}
\newcommand{\tb}{\tilde{b}}
\newcommand{\tc}{\tilde{c}}
\newcommand{\td}{\tilde{d}}
\newcommand{\tph}{\tilde{\phi}}
\newcommand{\tps}{\tilde{\psi}}

\newcommand{\hsp}{\hspace{\arraycolsep}} 
\newcommand{\eql}{\hsp=\hsp}
\newcommand{\slim}{\mbox{\rm s-lim}}
\newcommand{\tr}{\mathrm{trace}}

\newcommand{\twotwon}[4]{\left(\begin{array}{cc}#1&#2\\ #3&#4\end{array}\right)}
\newcommand{\twoc}[2]{\left(\begin{array}{c}#1\\#2\end{array}\right)}
\newcommand{\twor}[2]{\left(\begin{array}{cc}#1&#2\end{array}\right)}

\begin{document}

\date{}
\title{Asymptotic formulas for determinants of a special class of Toeplitz + Hankel matrices
\thanks{2010 MSC: 47B35; \ keywords: Toeplitz operator, Hankel operator, Toeplitz plus Hankel operator }}
\author{Estelle L. Basor\thanks{ebasor@aimath.org}\\
               American Institute of Mathematics\\
              Palo Alto, CA  94306, USA
        \and
        Torsten Ehrhardt\thanks{tehrhard@ucsc.edu}\\
        Department of Mathematics\\
        University of California\\
       Santa Cruz, CA 95064, USA}     
\maketitle

\begin{center}\em 
Dedicated to Albrecht B\"ottcher on the occasion of his sixteith birthday
\end{center}

\begin{abstract}
We compute the asymptotics of the determinants of certain $n\times n$ Toeplitz + Hankel matrices
$T_n(a)+H_n(b)$ as $n\to\infty$ with symbols of Fisher-Hartwig type. More specifically
we consider the case where $a$ has zeros and poles and where $b$ is related to $a$ in  specific ways.
Previous results of Deift, Its and Krasovsky dealt with the case where $a$ is even.
We are generalizing this in a mild way to certain non-even symbols.

\end{abstract}

\section{Introduction}

For many recent applications, an asymptotic formula for determinants of the sum of finite Toeplitz and Hankel matrices has been of interest.  For example,  if we let  $a$ be in $L^{1}(\T)$ and denote the $k$th Fourier coefficients of $a$ by $a_{k}$ then understanding the behavior of 
\[ \det \left(a_{j-k} + a_{j+k+1} \right) _{j,k  = 0, \dots ,n-1}\]
as $n\to\iy$ is important in random matrix theory. It has been shown in \cite{BE1}
that the above determinant behaves asymptotically like $G^n E$ with certain explicitly given constants $G$ and $E$ if $a$ is a sufficiently well-behaved function.
Such a result is an analogue of the classical Szeg\"o-Widom limit theorem \cite{W2} for Toeplitz determinants.

The above determinant is a special case of more general determinants, 
\[ \det \left(a_{j-k} + b_{j+k+1} \right) _{j,k  = 0, \dots ,n-1,}\]
where both $a$ and $b$ are in $L^{1}(\T)$. We refer to the functions $a$ and $b$ as symbols. 
The goal is to find the asymptotics in the case of well-behaved $a$ and $b$ and also for the singular Fisher-Hartwig type symbols (symbols, with say, jump discontinuities or zeros).
While an asymptotic formula in such a general case (with explicit description of the constants) is probably not doable, much recent progress has been made in some special cases. 

To be more precise a Fisher-Hartwig symbol is one of the form
\bq \label{eqn1}
a(e^{i \theta}) = c(e^{i\theta}) \prod_{r = 1}^{R}  v_{\tau_{r}, \al_{r}} (e^{i\theta}) u_{\tau_{r}, \beta_{r}}(e^{i\theta}) 
 \eq
where $c$ is a sufficiently well-behaved function (i.e., sufficiently smooth, nonvanishing, and with winding number zero) and
for $\tau = e^{i \phi}$,
 \begin{align*}
u_{\tau,\beta}(e^{i\theta}) &= \exp(i\beta(\theta-\phi-\pi)), \quad
0< \theta-\phi <2\pi,
\\[.5ex]
 v_{\tau, \al}(e^{i\theta}) & =  ( 2 - 2 \cos (\theta-\phi))^{\al}.
 \end{align*}
The symbol $u_{\tau, \beta}$ has a jump at the point $\tau$ on the unit circle and the function $v_{\tau. \al}$ can be singular (say if $\al$ has negative real part,) or be zero at $\tau$. We will generally refer to this last factor as having a singularity of ``zero'' type.
Furthermore, $\alpha_r$ and $\beta_r$ are complex parameters (where we assume $\Re \alpha_r>-1/2$) and $\tau_1,\dots,\tau_R$ are distinct points on the
 unit circle $\T$.

In the case of smooth symbols, we cite the results in \cite{BE4} where the case of 
\[ 
\det \left(a_{j-k} + b_{j+k+1} \right) _{j,k  = 0, \dots ,n-1}
\]
with $b(e^{i\theta})=\pm e^{i\ell \theta} a(e^{i\theta})$ and $\ell$ fixed is considered.  
It is worth mentioning that among those cases, there are four special cases of particular interest,
\begin{itemize}
\item[(i)]
$b(e^{i\theta})=a(e^{i\theta})$,
\item[(ii)]
$b(e^{i\theta})=-a(e^{i\theta})$,
\item[(iii)]
$b(e^{i\theta})=e^{i\theta}a(e^{i\theta})$,
\item[(iv)]
$b(e^{i\theta})=-e^{-i\theta} a(e^{i\theta})$,
\end{itemize} 
in which the asymptotics have the form $ G^n E$ with non-zero $E$. In the case of even symbols, i.e., $a(e^{i\theta})=a(e^{-i\theta})$, these four cases are also related to the random matrices 
taken from the classical groups \cite{BR,FF}. Furthermore, these Toeplitz+Hankel determinants are expressable as Hankel determinants as well.

In the case of $b=a$, two earlier papers of the authors consider the case of jump discontinuities \cite{BE1, BE2}. 
Furthermore, in the above four cases where in addition $a$ is even, 
the results of Deift, Its, and Krasovsky \cite{DIK} are quite complete and impressive. They allow quite general Fisher-Hartwig symbols with both zeros and jumps.
In \cite{DIK} 
the asymptotics are of the form
\[ G^{n}\,n^{p}\,E\]
where $G, p,$ and $E$ are explicitly given constants.
However, none of the earlier mentioned papers cover the case where the symbol is allowed to be non-even and with singularities of the zero type. So this is the focus of this paper, a non-even symbol with certain specified types of Fisher-Hartwig symbols.  We prove an asymptotic formula of the same form as above. This is a step in the ultimate goal of asymptotics for non-even symbols with general Fisher-Hartwig singularitites.

In order to briefly sketch the main ideas of the paper, let $T(a)$ and $H(b)$ stand for the Toeplitz and Hankel operators with symbols $a$ and $b$ 
acting on $\ell^2$, and let $P_n$ stand for the finite section projection on $\ell^2$. The precise definition of these operators will be given
in the next section. The above determinants can be understood as the determinant of 
$P_n (T(a)+H(b)) P_n$ for certain symbols $a,b$ of Fisher-Hartwig type. Since we will allow not only for jumps, but also for zeros and poles the underlying operator (or its inverse) is generally not bounded. Hence the first step is to reformulate the problem as one which involves only bounded operators.
This will be done by establishing an identity of the kind
\begin{equation}\label{f.id1}
\det P_n(T(a) +H(b))P_n = \det \Big( P_n  T^{-1}(\psi)(T(c)+ H(cd\phi)) T\iv(\psi\iv) P_n\Big) 
\end{equation}
where the functions $c$ and $d$ are smooth and $\phi$ and $\psi$ have only jump discontinuities. The next major step is a separation theorem, 
which allows to ``remove'' the smooth functions $c$ and $d$, i.e., 
$$
\frac{\det  \Big( P_n T^{-1}(\psi)(T(c)+ H(cd\phi)) T\iv(\psi\iv) P_n\Big)}%
{\det \Big( P_n T^{-1}(\psi)(I+ H(\phi)) T\iv(\psi\iv) P_n\Big)  }
\sim E\cdot G^n,\qquad n\to\iy,
$$
with explicit constant $E$ and $G$. Then by using the first identity again we relate the last determinant back to a Toeplitz+Hankel determinant,
$$
\det P_n(T(a_0) +H(b_0))P_n = \det  \Big( P_n T^{-1}(\psi)(I+ H(\phi)) T\iv(\psi\iv) P_n\Big)\,,
$$
where, as it turns out, $a_0$ and $b_0$ are Fisher-Hartwig symbols which in their product do not have the smooth part.
Of course, the whole procedure is only as useful as far as we are able to obtain the asymptotics of $\det P_n(T(a_0) +H(b_0))P_n$.
Here we apply the results of Deift, Its, and Krasovsky mentioned above \cite{DIK} to identify this asymptotics in four special cases. 
The relation between our symbols $a$ and $b$, and the symbols $a_0$ and $b_0$ to which we apply \cite{DIK}, is in fact given by
$$
a= c\, a_0,\qquad b= c\, d\,  b_0,
$$
where $a_0$ is even and of Fisher-Hartwig type and $b_0$ relates to $a_0$ as in (i)--(iv). The functions $c$ and $d$, while required to be sufficiently well-behaved, do not have to be even. The function $d$ is required to satisfy $d(e^{i\theta})d(e^{-i\theta})=1$, $d(\pm1)=1$. In this sense, our results generalize some of the results of \cite{DIK}.

Note that in the most general case (for which we are able to do the separation theorem), the asymptotics of the corresponding 
$\det P_n(T(a_0) +H(b_0))P_n $ is not known.

The idea for establishing an identity of the kind \eqref{f.id1} and proving a separation theorem is due to B\"ottcher and Silbermann. In 1985 they proved the 
 Fisher-Hartwig conjecture for symbols with small parameters (i.e, symbols \eqref{eqn1} with  $|\Re \alpha_r|<1/2$ and $|\Re \beta_r|<1/2$), which was considered a major breakthrough at the time (see \cite{BS2} and \cite[Sect.~10.10]{BS}).
Although they  considered bounded invertible  operators acting between different weighted $L^2$-spaces, the essential point of their analysis can be expressed as identity 
of the kind
$$
\det T_n(a) = \det  P_n T\iv (\psi) T(\phi) T\iv(\psi\iv) P_n \,,
$$
where $a$ is a Fisher-Hartwig symbol with jumps and zeros/poles, while $\phi$ and $\psi$ are Fisher-Hartwig symbols with jumps only
(the corresponding  Toeplitz operators being bounded and invertible under certain conditions).

Here is an outline of the paper. We begin with the operator theoretic preliminaries. This is done is section 2. In section 3, we reformulate the problem so that we can consider determinants of bounded operators only.  In section 4, some additional operator theoretic results are given that are particularly useful for our situation. We then prove, in section 5,  a ``separation'' theorem, that is, a theorem that allows us to compute the asymptotics from a combination of the smooth symbols and some specific cases of singular symbols where the results can be computed by other means. This is done in more generality than is needed for our final results, but it may prove to be useful in the future if other specific cases of singular symbols are obtained. 

Section 6 is devoted to infinite determinant computations that are required to describe constants explicitly and the next section contains the known results for the specific known singular symbols. Everything is collected in section 8 where the final asymptotics are computed.

Finally, the last section contains some additional results. In the course of the computations for the main results of this paper, we discovered that the inverse of certain Identity plus Hankel operators had inverses that could be described using Toeplitz operators, their inverses and Hankel operators. So the inverse expressions may be of independent interest and are also included.


\section{Preliminaries}

We denote by $\ell^{2}$ the space of all complex-valued sequences $\{x_{n}\}_{n=0}^{\infty}$ quipped with usual  $2$-norm.  The set $\cL(\ell^{2})$ is the set of bounded operators on  $\ell^{2}$ and $ \cC_1(\ell^2)$ is the set of trace class operators on $\ell^{2}$.

The Toeplitz operator $T(a)$ and Hankel operator $H(a)$ with symbol $a\in L^{\infty}(\T)$ are the bounded linear operators defined on $\ell^{2}$ with matrix representations
$$
T(a) = (a_{j-k}),\,\,\,\,\, 0\leq j,k < \infty, 
$$
and
$$
H(a) = (a_{j+k+1}),\,\,\,\,\, 0\leq j,k < \infty.
$$
%
It is well-known and not difficult to prove that Toeplitz and Hankel operators satisfy the fundamental identities
\be\label{T1} T(ab) = T(a)T(b) + H(a)H(\tilde b)
\ee
and
\be\label{H1}
H(ab) = T(a)H(b) + H(a)T(\tilde b) .
\ee
In the last two identities and throughout the paper we are using the notation
$$
\tilde b (e^{i\theta}) := b(e^{-i\theta}) .
$$
It is worthwhile to point out that these identities imply that 
\be\label{T2}
T(abc)=T(a)T(b)T(c),\qquad H(ab\tilde{c})=T(a)H(b)T(c)
\ee
for $a,b,c,\in L^\iy(\T)$ if $a_n=c_{-n}=0$ for all $n>0$.

We define the (finite section) projection $P_{n}$  by
\[ 
P_{n} : \{x_k\}_{k=0}^\iy\in \ell^2 \mapsto \{y_k\}_{k=0}^\iy\in \ell^2,\qquad
y_k=\left\{\ba{cl} x_k & \mbox{if } k<n\\ 0 & \mbox{if } k\ge n  \,. \ea\right.
\]
Using $P_n$  we can view our determinants of interest as determinants of truncations of infinite matrices, 
\[ 
\det ( T_n(a) + H_n(b))=
P_{n} (T(a) + H(b) ) P_{n}. \]
For bounded $a$ and $b$ this is the truncation of a sum of bounded operators, but  even more generally for $a,b\in L^1(\T)$ providing we view the operators as being defined on the space of sequences with only a finite number of non-zero terms.

In the next sections we will be mostly concerned with functions $a$ that are products of continuous functions times those with certain specific types of singularities. It will be convenient for the continuous function factors to satisfy certain properties. To describe this,
we consider the Banach algebra called the  Besov class $B_{1}^{1}.$ This is the algebra of all functions $a$ defined on the unit circle for which
\[
\|a\|_{B_{1}^{1}} := \int_{-\pi}^{\pi}\frac{1}{y^{2}}\int _{-\pi}^{\pi}|a(e^{ix+iy}) + a(e^{ix-iy}) -2a(e^{ix})|\,dx\, dy < \iy .
\]
A function $a$ is in $B_1^1$ if and only if the Hankel operators $H(a)$ and $H(\tilde{a})$ are both trace class. Moreover, the Riesz projection is bounded on $B_1^1$, and an equivalent norm is given by
\[ |a_{0}| + \|H(a)\|_{\cC_1} + \|H(\tilde{a})\|_{\cC_1},\]
where  $\|A\|_{\cC_1}$ is the trace norm of the operator $A$.

Let us also recall the notion of Wiener-Hopf factorization. There are several versions of it. We say that 
$c\in L^\iy(\T)$ has a bounded (canonical) factorization if we can write $c=c_-c_+$
with $c_+,c_+\iv \in H^\iy_+(\T)$ and $c_-,c_-\iv \in H^\iy_-(\T)$, where 
$$
H^\iy_\pm (\T)=\{f\in L^\iy(\T) \,:\, f_n=0 \mbox{ for all } \pm n<0\,\}.
$$
We say that $c\in B^1_1$ has a canonical factorization in $B_1^1$ if 
we can write $c=c_-c_+$ with $c_+,c_+\iv \in H^\iy_+(\T)\cap B_1^1$ and $c_-,c_-\iv \in H^\iy_-(\T)\cap B_1^1$.
It is well known (see, e.g., \cite[Sect. 10.24]{BS}) that $c$ admits a canonical factorization in $B_1^1$ if and only if
the function $c$ does not vanish on $\T$  and has winding number zero. In this case, the logarithm exists, $\log c\in B^1_1$, 
and one can define normalized factors,
\bq\label{f.c+-}
c_\pm(t) = \exp\left( \sum_{k=1}^\iy t^{\pm k} [\log c]_{\pm k}\right),
\eq
which yield a factorization $c=c_- G[c] c_+$ with the constant 
\bq\label{f.G[c]}
G[c]:= \exp( [\log c]_0)
\eq
representing the geometric mean.

For our purposes it is also important to consider a factorization of the kind
\bq\label{f.dd+}
d= \td_+\iv d_+ \quad \mbox{ with } d_+,d_+\iv \in B_1^1\cap H^\iy_+(\T),
\eq
in which the ``minus'' factor $\td_+\iv $ is given by the ``plus'' factor $d_+$. It is not too difficult to show (using the above result and the uniqueness of factorization up to multiplicative constants) that $d\in B_1^1$ possesses a factorization of the above kind if and only if $d$ does not vanish on $\T$, has winding number zero and satisfies the conditions
$d\td=1$ and $d(\pm 1)=1$. Notice that in this case $\log d\in B_1^1$ is an odd function and thus $G[d]=1$.


\section{Reformulating the problem}
\label{s3}

As described in the introduction, we are interested in determinants of Toeplitz plus Hankel matrices with singular symbols. 
Let us denote the corresponding (infinite) operator by 
$$
M(a,b) := T(a)+H(b).
$$
Notice that when the symbols involve zeros or poles, then either $M(a,b)$ or its inverse are in general not bounded operators anymore.
The purpose of this section is to reformulate the problem about the asymptotics of $\det (P_n M(a,b) P_n)$
as one for $\det (P_n A P_n)$ where $A$ is a bounded (and invertible) operator on $\ell^2$. 
%
%
%
%
More precisely, we are going to prove a formula
\begin{align}\label{f.M-A}
\det P_n M(a,b)P_n = \det P_n T\iv(\psi) M(c,c \,d\,\phi ) T\iv(\psi\iv) P_n ,
\end{align}
where $a$ and $b$ are certain functions of Fisher-Hartwig type (allowing in particular for zeros and jumps)
while on the right hand side $\psi$ and $\phi$ are functions with jump discontiniuties only and with ranges of parameters such that 
$T(\psi)$ and $T(\psi\iv)$ are invertible Toeplitz operators. The functions $c$ and $d$ are smooth and nonvanishing functions with winding number zero.

Since the above formula involves inverses of Toeplitz operators, let us first recall a well-known sufficient invertibility criterion
(see, e.g., \cite{W} or \cite{BS}).

\begin{theorem}\label{t.T.inv}
Let $c$ be a continuous and nonvanishing function on $\T$ with winding number zero, let 
$\tau_1,\dots,\tau_R\in\T$ be distinct, and  
$$
\psi (e^{i\theta})  = c (e^{i\theta})   \prod_{r=1}^R u_{\tau_r,\beta_r}(e^{i\theta}) \,.
$$
If  $|\Re\beta_r|<1/2$ for all $1\le r\le R$, then $T(\psi)$ is invertible on $\ell^2$.
\end{theorem}

Let us now introduce the functions for which identity \eqref{f.M-A} will be proved.
For these functions the separation theorem will be proved later on as indicated in the introduction.
We consider
\begin{align}
a &= c\, v_{1,\alpha^+}\;   v_{-1,\alpha^-}\,
\prod_{r=1}^R v_{\tau_r,\alpha^+_r}\;  v_{\bar{\tau}_r,\alpha^-_r}\;,
\\
b &= c\,d\, v_{1,\alpha^+}\;  u_{1,\beta^+} \; v_{-1,\alpha^-}\; u_{-1,\beta^-}\;
\prod_{r=1}^R v_{\tau_r,\alpha_r}u_{\tau_r,\beta_r}\;  v_{\bar{\tau}_r,\alpha_r} u_{\bar{\tau}_r,\beta_r}\,.
\end{align}
The functions $c$ and $d$ are smooth nonvanishing functions with winding number zero.
In addition, we will require that  $d\tilde{d}=1$ and $d(\pm1)=1$. 
We also assume that $\tau_1,\dots,\tau_R\in\T_+$ are distinct,
where
$$
\T_+:=\left\{\, t\in \T\,:\, \Im(t) >0\,\right\}\,,
$$
and that 
$$
\alpha^\pm,\;\;
\beta^\pm,\;\;
\alpha_r^\pm,\;\;
\beta_r
$$
are complex parameters satisfying the conditions \eqref{f.16} and \eqref{f.17} stated below, whereas
\begin{equation}
\alpha_r := \frac{\alpha_r^++\alpha_r^-}{2}\qquad \mbox{ for } 1\le r \le R.
\end{equation}

The functions $\psi$ and $\phi$ that will appear in the identity are 
\begin{align}
\label{f.psi}
\psi &= u_{1,\alpha^+}\;   u_{-1,\alpha^-}
\prod_{r=1}^R u_{\tau_r,\alpha^+_r}\;  u_{\bar{\tau}_r,\alpha^-_r}\;,
\\
\label{f.phi}
\phi  &=   u_{1,\gamma^+} \; u_{-1,\gamma^-}\;
\prod_{r=1}^R u_{\tau_r,\gamma_r}\;   u_{\bar{\tau}_r,\gamma_r}\;,
\end{align}
where
\begin{equation}\label{f.gam}
\gamma^\pm:=\alpha^\pm +\beta^\pm\,,\qquad
\gamma_r:=\alpha_r +\beta_r\,.
\end{equation}
The restrictions which we are going to impose on the parameters are the following:
\begin{align}\label{f.16}
|\Re\alpha^\pm|<1/2,\qquad
|\Re\alpha^\pm_r|<1/2,
\end{align}
which guarantee the invertibility of $T(\psi)$ and $T(\psi\iv)$, and
\begin{align}\label{f.17}
-3/2<\Re \gamma^+<1/2,\qquad
-1/2<\Re \gamma^-<3/2,\qquad
|\Re\gamma_r|<1/2.
\end{align}
The last conditions are needed later on.

\begin{theorem}\label{t4.1}
Let $a,b,c,d,\phi,\psi$ be as above with \eqref{f.16} being assumed. Then
$$
\det P_n M(a,b) P_n =\det \Big( P_n T\iv(\psi)M( c, c\,d\, \phi)T\iv(\psi\iv) P_n\Big).
$$
\end{theorem}
\begin{proof}
We first notice that $a,b\in L^1(\T)$. Hence the $P_n M(a,b) P_n$ is a well-defined matrix, although $M(a,b)$ may be an unbounded operator.
The proof of the identity is based on the Wiener-Hopf factorization of the underlying generating functions.
In order to avoid unbounded factors, let us assume for the time being that all the parameters $\alpha^\pm,\beta^\pm,\alpha_r^\pm,\beta_r$ are purely imaginary.
The general case follows by observing that both sides of the identity are analytic in each of these parameters.

In order to obtain the factorization introduce the functions 
$$
\eta_{\tau,\gamma}(t)=(1-t/\tau)^{\gamma},\qquad
\xi_{\tau,\delta}(t)=(1-\tau/t)^{\delta},
$$
where the branches of $\eta$ (analytic inside the unit circle) and $\xi$ (analytic outside the unit circle)  are chosen so that 
$\eta_{ \tau, \gamma}(0) = \xi_{\tau, \delta}(\iy) =1$. Using the above definitions we can produce the well-known Wiener-Hopf factorizations for 
$$
u_{\tau, \beta} = \xi_{\tau, -\beta}\,\eta_{\tau, \beta}\,,\qquad\quad
v_{\tau,\alpha}=\xi_{\tau,\alpha}\, \eta_{\tau,\alpha}\,.
$$

Now put
\begin{align*}
\psi_-  &= \xi_{1,-\alpha^+}\,\xi_{-1,-\alpha^-} 
\prod_{r=1}^R \xi_{\tau_r,-\alpha^+_r}\;  \xi_{\bar{\tau}_r,-\alpha^-_r}\;,
\\
\psi_+  &= \eta_{1,\alpha^+}\;\;\,   \eta_{-1,\alpha^-}\;\,
\prod_{r=1}^R \eta_{\tau_r,\alpha^+_r}\;\,  \eta_{\bar{\tau}_r,\alpha^-_r}\;.
\end{align*}
Then, indeed, $\psi=\psi_-\psi_+$. Furthermore,
\begin{align}
\psi_-\iv \psi_+ &= v_{1,\alpha^+}\;   v_{-1,\alpha^-}\,
\prod_{r=1}^R  v_{\tau_r,\alpha^+_r}\;  v_{\bar{\tau}_r,\alpha^-_r},
\nonumber
\\
\psi_-\iv \tps_+ &= \xi_{1,2\alpha^+}\;   \xi_{-1,2\alpha^-}\,
\prod_{r=1}^R  \xi_{\tau_r,\alpha^+_r+\alpha_r^-}\;  \xi_{\bar{\tau}_r,\alpha^+_r+\alpha^-_r}.
\nonumber
\end{align}
Here notice that $\tilde{\eta}_{\tau,\alpha}=\xi_{\bt,\alpha}$.
The latter can be written as the product of 
$$
v_{1,\alpha^+}\;   v_{-1,\alpha^-}\,
\prod v_{\tau_r,\alpha_r}\;  v_{\bar{\tau}_r,\alpha_r}
\quad\mbox{ and }\quad
u_{1,-\alpha^+}\; u_{-1,-\alpha^-}\,
\prod u_{\tau_r,-\alpha_r}\;  u_{\bar{\tau}_r,-\alpha_r}
$$
as $\alpha_r=(\alpha_r^++\alpha_r^-)/2$.
Thus we see that 
$$
\psi=\psi_-\psi_+\;, \quad 
a= c\,\psi\iv_- \,\psi_+\;,\quad
b=c\, d\, \phi\, \psi_-\iv\, \tps_+\;.
$$

It follows that 
$$
\det P_{n} M(a, b) P_{n} =   \det P_{n} M( \psi_{-}^{-1}\,c\,\psi_{+}, \psi_{-}^{-1} c\, d\, \phi\, \tps_{+}) P_{n}$$ 
which equals 
$$\det P_{n} T(\psi_{-}^{-1})
M( c, c \, d\, \phi ) T(\psi_{+}) P_{n}
$$
by using \eqref{T2}.
Also, notice that the determinants of $P_{n}T(\psi_{\pm})P_{n}$ and $P_{n}T(\psi_{\pm}^{-1})P_{n}$ are one since they are either upper or lower triangular matrices with ones on the diagonal. Using this and the observation that  
$$
P_n T(\psi_+\iv)P_n = P_n T(\psi_+\iv),\qquad
P_n T(\psi_-) P_n = T(\psi_-)P_n,
$$
the above equals
$$
\det P_{n} T(\psi_{+}^{-1}) T(\psi_{-}^{-1}) M(c, c\, d  \, \psi) T(\psi_{+}) T(\psi_{-}) P_{n}. 
$$
Applying the following formulas
for the inverses,
$$
T\iv(\psi)=T(\psi_+\iv)T(\psi_-\iv),\qquad
T\iv (\psi\iv)= T(\psi_+)T(\psi_+),
$$
concludes the proof of the identity.
\end{proof}

It is interesting to consider certain special 
 cases. What we have in mind is the case where the Fisher-Hartwig part of $a$ (i.e.,
the product without the function $c$) is even. This happens if
$$
\alpha_r^+=\alpha_r^-=\alpha_r\,.$$
If in addition, we put $\beta_r=0$, then
$$
b=  u_{1,\beta^+}\,u_{-1,\beta^-}\, d \,a  \,.
$$
There are four specific choices of parameters $\beta^\pm$ where the factor $\phi_0:=u_{1,\beta^+}\,u_{\beta^-}$ is actually continuous:
\begin{itemize}
\item[(1)]
$\beta^+=\beta^-=0$, $\phi_0(t)=1$\;;
\item[(2)]
$\beta^+=-1$, $\beta^-=1$, $\phi_0(t)=-1$\;;
\item[(3)]
$\beta^+=0$, $\beta^-=1$, $\phi_0(t)=t$\;;
\item[(4)]
$\beta^+=-1$, $\beta^-=0$, $\phi_0(t)=-1/t$\;.
\end{itemize}
Notice that the conditions \eqref{f.16} and \eqref{f.17} on the parameters $\alpha^\pm$ and $\alpha_r$ amount to the following
\begin{align}\label{f.16x}
|\Re\alpha^\pm|<1/2,\qquad
|\Re\alpha_r|<1/2.
\end{align}
To summarize, in these special cases we have
\begin{align}
a &= c\, v_{1,\alpha^+}\;   v_{-1,\alpha^-}\,
\prod_{r=1}^R v_{\tau_r,\alpha_r}\;  v_{\bar{\tau}_r,\alpha_r}\;,
\qquad
b=  \phi_0\, d\, a.
\end{align}
Notice that the cases (1)-(4) correspond to the cases (i)-(iv) considered in the introduction, but are slightly more general due to the factor $d$.
The reason why we single out these four special cases, is because for the computations that are
made later in this paper, it is in these cases that we can actually determine the asymptotics,
whereas in the more general case we can only reduce the asymptotics to a simplified determinant problem for which an answer is unknown.

\section{Additional operator theoretic results}
We need some results about Toeplitz operators and Hankel
operators (see \cite{BE1} and 
\cite{BS}  for the general theory).
First of all, in addition to the projections $P_n$,
and $Q_n=I-P_n$ we define
\bq
W_n (f_{0},f_{1},\dots)  &=& (f_{n-1},f_{n-2},\dots 
,f_{1},f_{0},0,0,\dots),\nn\\  
V_n (f_{0},f_{1},\dots)  &=& 
(0,0,\dots,0,0,f_{0},f_{1},f_{2},\dots),\nn\\
V_{-n}(f_{0},f_{1},\dots)&=& (f_{n},f_{n+1},f_{n+2},\dots).\nn
\eq
It is easily seen that
$W_n^2=P_n$, $W_n=W_nP_n=P_nW_n$, $V_nV_{-n}=Q_n$ and $V_{-n}V_n=I$.
Note also that
\be\label{f1.THVW}
P_nT(a)V_n \eql W_nH(\tilde{a}), \qquad
V_{-n}T(a)P_n \eql H(a)W_n.
\ee
Moreover, we have
\be\label{f1.VVWW}\\
V_{-n} T(a)V_n = T(a),\quad 
V_{-n}H(a) = H(a)V_n,\quad
W_nT(a)W_n = P_nT(\tilde{a})P_{n}.
\ee

In the proofs that follow we will need the notions of stability and strong convergence and
we describe those now. 

Let $A_n$ be a sequence of operators.  This sequence is said
to be stable if there exists an $n_0$ such that the operators $A_n$
are invertible for each $n\geq n_0$ and ${\sup}_{n\geq n_0}
\|A_{n}^{-1}\|<\infty$.  Moreover, we say that $A_n$ converges
strongly on $\ell^{2}$ to an operator $A$ as $n\to\iy$ if $A_nx\to Ax$ in
the norm of $\ell^{2}$ for each $x\in\ell^{2}$.  When dealing with finite
matrices $A_n$, we identify the matrices and their inverses with
operators acting on the image of $P_{n}$.  It is well known (see \cite[Th.~4.15]{BS} and worthy to note that
stability is related to strong convergence of the inverses
(and their adjoints) in the following sense. 

\begin{lemma}\label{l1.1}
Suppose that $A_n$ is a stable sequence such that $A_n\to A$ and
$A_n^*\to A^*$ strongly.  Then $A$ is invertible, and
$A_n\iv\to A\iv$ and $(A_n\iv)^*\to (A\iv)^*$ strongly.
\end{lemma}

Recall that for trace class operators, the trace ``$\tr \, A$'' and the
operator determinant ``$\det(I+A)$'' are well defined and continuous
with respect to $A$ in the trace class norm.  The following well known result
shows the connection with strong convergence.

\begin{lemma}\label{l1.1b}
Let $B$ be a trace class operator and suppose that $A_n$ and $C_n$
are sequences such that $A_n\to A$ and $C_n^*\to C^*$ strongly. Then 
$A_nBC_n\to ABC$ in the trace class norm.
\end{lemma}

We can use the first lemma to obtain information about the strong convergence of the inverses of Toeplitz matrices.

\begin{proposition}\label{p1.2}
Let $\phi\in L^\iy(\T)$. If $T_n(\phi)$ is stable, then $T(\phi)$
is invertible and
$$
T\iv(\phi) = \slim\; T_n\iv(\phi), 
\qquad
T\iv(\wt{\phi}) = \slim\; W_nT_n\iv(\phi)W_n.     
$$
\end{proposition}
\begin{proof}
Since $P_n^*=P_n\to I$ strongly, it follows that $T_n(\phi)\to T(\phi)$ strongly and the same holds for the adjoints.
Lemma \ref{l1.1} implies the first statement. For the second one, observe that $W_n T_n(\phi) W_n=T_n(\wt{\phi})$ and proceed similarly.
Also note that $T(\wt{\phi})$ is the transpose of $T(\phi)$, thus also invertible.
\end{proof}

We will need a new definition and additional results about strong convergence in what follows (see also  \cite[Thm.~7.13]{BS}).
Let $\cA$ equal the set of all bounded operators $A$ defined on $\ell^{2}$ such that the operator
\[  
\twoc{W_{n}}{ V_{-n} } A \twor {W_{n}}{ V_{n} } = \twotwon{W_{n}AW_{n}}{W_{n}AV_{n}}{V_{-n}AW_{n}}{V_{-n}AV_{n}}
\]
 along with its adjoint (which replaces $A$ with $A^{*}$) converge strongly to operators defined on $\ell^{2} \oplus \ell^{2}$.
 In other words 
 $$
\pi(A) :=\slim \left( \twoc{W_{n}}{ V_{-n} } A \twor {W_{n}}{ V_{n} } \right) 
$$
exists.
 
 \begin{lemma}\label{3.SC}
The set  $\cA$ is a (closed in the operator topology) $C^{*}$-subalgebra of $L(\ell^2)$, and the map
$\pi:\cA\to L(\ell^2\oplus\ell^2)$  is a *-homomorphism.
 \end{lemma}
 \begin{proof}
 It is easy to see that the sum, the product and the involution are closed operations in $\cA$ and at the same time that $\pi$ is a *-homomorphism.
 Using that the norms of $W_n$ and $V_{\pm n}$ are one, one can conclude the map $\pi$ is bounded. The fact that 
 $\cA$ is closed can be shown straightforwardly using a Cauchy sequence argument (see also \cite[Thm.~7.13]{BS}).
 %
 \end{proof}

We now relate Toeplitz and Hankel operators to $\cA$ and $\pi$. 

\begin{lemma}\label{2.4}
For $\phi$ in $L^{\infty}(\T)$ the operators $T(\phi)$ and $H(\phi)$ belong to $\cA$. Moreover,
$$
\pi(T(\phi))=\twotwon{ T(\tilde{\phi}) }{ H(\tilde{\phi})}{  H(\phi)}{T(\phi)},\qquad
\pi(H(\phi))=\twotwon{0}{0}{0}{0}.
$$
\end{lemma}
\proof
We consider first the Toeplitz operator. We use the identities
\begin{align*}
W_n T(\phi) W_n &= P_n T(\tilde{\phi})P_n, &\quad W_n T(\phi) V_n &= P_n H(\tilde{\phi}),\\
V_{-n} T(\phi) W_n &=  H(\phi) P_n ,&  \quad V_{-n} T(\phi) V_n &= T(\phi),
\end{align*}
which we stated at the beginning of the section, to show the strong convergence.
For the Hankel operator we consider 
\begin{align*}
W_n H(\phi) W_n &= W_n V_{-n} T(\phi)P_n , &\quad W_n H(\phi) V_n &= W_n V_{-n} H(\phi),\\
V_{-n} H(\phi) W_n &=  V_{-2n} T(\phi) P_n ,&  \quad V_{-n} H(\phi) V_n &= V_{-2n}H(\phi),
\end{align*}
and the strong convergence follows because $V_{-n}\to 0$ strongly. For the adjoints the argumentation is analogous.
\qed

If we abbreviate
$$
L(\phi):= \twotwon{ T(\tilde{\phi}) }{ H(\tilde{\phi})}{  H(\phi)}{T(\phi)},
$$
then  \eqref{T1} and \eqref{H1} imply that 
\bq\label{f.L12}
L(\phi_1\phi_2)=L(\phi_1)L(\phi_2)\,.
\eq
This is not surprising since by an appropriate identification of $\ell^2\oplus \ell^2$ with $\ell^2(\Z)$ it is easily seen that $L(\phi)$ is 
the Laurent operator with symbol $\phi$.

The following result is what we will need in the next section. Notice that in the case of $c=1$ we have that 
$\pi(A)$ is the identity operator on $\ell^2\oplus \ell^2$.

\begin{proposition}\label{p2.5}
Let $A=T\iv(\psi)(T(c)+H(\phi))T\iv(\psi\iv)$ where
$c,\phi,\psi\in L^\iy(\T)$ are such that $T(\psi^{\pm 1})$ are invertible. Then
 $A\in \cA$ and 
$$
\pi(A)=\twotwon{T(\tc)}{H(\tc)}{H(c)}{T(c)}.
$$
\end{proposition}
\proof
Since $\cA$ is a C*-algebra (hence inverse closed) and $\pi$ is a *-homomorphism, it follows that $T\iv(\psi)\in\cA$ and 
$$
\pi(T\iv(\psi))=(\pi(T(\psi)))\iv = \twotwon{T(\tps) }{ H(\tps)}{ H(\psi)}{T(\psi)}^{-1}
=\twotwon{T(\tps\iv) }{ H(\tps\iv)}{ H(\psi\iv)}{T(\psi\iv)}=L(\psi\iv)\,.
$$
The inversion of the operator matrix follows from \eqref{T1} and \eqref{H1}.
Similarly, we obtain 
$$
\pi(T\iv(\psi\iv)) = \twotwon{T(\tps) }{ H(\tps)}{ H(\psi)}{T(\psi)}=L(\psi)
$$ and 
$\pi(T(c)+H(\phi))=L(c)$.
Using \eqref{f.L12} we obtain 
$\pi(A)=L(\psi\iv) L(c) L(\psi) = L(c)$, which is the formula for $\pi(A)$.
\qed


\section{Separation theorems}

We now establish a separation theorem, which we are formulation in a quite general setting.

\begin{theorem}\label{t5.1}
Let $\psi,\phi\in L^\iy(\T)$ with $\phi\tilde{\phi}=1$  be such that $T(\psi)$ is invertible on $\ell^2$
and such that the sequence 
\be\label{f.An}
A_n=P_n T\iv(\psi) M(1,\phi) T\iv(\psi\iv) P_n
\ee
is stable. Moreover, assume that 
$c\in B^1_1$ is nonvanishing and has winding number zero and that $d\in B_1^1$ has a Wiener-Hopf factorization $d=d_+\td_+\iv$ in $B^1_1$.
Then 
$$
\lim_{n\to\infty}
\frac{
\det\Big( P_{n} T\iv(\psi) M(c, c d \phi )  T\iv(\psi^{-1}) P_{n}\Big)}{
G[c]^n  \det \Big( P_{n}\,T^{-1}(\psi) M(1, \phi ) T\iv(\psi^{-1})P_n\Big) }
 = E,
 $$
where $G[c]=\exp([\log c]_0)$ and
\begin{eqnarray}\label{E.const}
E&=&\det\Big( T\iv(\tc \td_+) T(\tc) T(\td_+)\Big) \times \det \Big( T(cd_+)T(c\iv d_+\iv)\Big) \times 
\\
\nonumber
&&\det \Big( T\iv(c d_+) T\iv(\psi) M(c,cd\phi) T\iv(\psi\iv) T(d_+) T(\psi\iv) M\iv(1,\phi)T(\psi)\Big).
\end{eqnarray}
\end{theorem}
\begin{proof}
We note that the conditions on $c$ and $d_+$ imply the invertibility of $T(cd_+)$ and $T(\tc\td_+)$ and the stability of $T_n(cd_+)$. 
Because $T(\psi)$ is invertible (and hence $\psi\iv \in L^\iy(\T)$) and one can conclude that $T(\tilde{\psi}\iv)$ is invertible.
Indeed,  the formula
$$ 
T\iv(\tilde{\psi}\iv)= T(\tilde{\psi})-H(\tilde{\psi})T\iv(\psi)H(\psi).
$$
can be verified straightforwardly using \eqref{H1} and \eqref{T2}.
Note that $T(\tilde{\psi}\iv)$ is the transpose of $T(\psi\iv)$, which thus is also invertible.
Furthermore the stability of $A_n$ implies the invertibility of $T\iv(\psi) M(1,\phi) T\iv(\psi\iv)$.
Hence $M(1,\phi)$ is invertible.
From the proof below it will follow that the operator determinants in \eqref{E.const} are well-defined,
by which we mean that the underlying operator is identity plus a trace class operator.

We start by looking at $M(c,cd\phi)$ modulo trace class operators. It equals
\begin{align*}
T(c)+H(cd_+\td_+\iv \phi) 
&=(T(c d_+)+H(cd_+\phi))T(d_+\iv)
\\
&=\left(T(cd_+)+T(cd_+)H(\phi)+H(cd_+)T(\tph)\right)T(d_+\iv)
\\
&=T(cd_+)M(1,\phi) T(d_+\iv)+\mbox{ trace class.}
\end{align*}
Hence modulo trace class, $T\iv(\psi) M(c,cd\phi) T\iv(\psi\iv)$ equals
$$
T\iv(\psi)  T(cd_+)M(1,\phi) T(d_+\iv) T\iv(\psi\iv).
$$
Since the commuators $[T\iv(\psi),T(cd_+)]$ and $[T(d_+\iv),T\iv(\psi\iv)]$ are trace class, it follows that 
$$
T\iv(\psi) M(c,cd\phi) T\iv(\psi\iv)
= T(cd_+) T\iv(\psi) M(1,\phi) T\iv(\psi\iv) T(d_+\iv)+K_1
$$
with a certain trace class operator $K_1$.
Now multiply with $P_n$ from the left and the right hand side and write
$$
P_n T(cd_+) A T(d_+\iv) P_n
=
P_n T(cd_+) P_n A P_n T(d_+\iv) P_n +
P_n T(cd_+) Q_n A Q_n T(d_+\iv) P_n +
$$
$$
P_n T(cd_+) Q_n A P_n T(d_+\iv) P_n +
P_n T(cd_+) P_n A Q_n T(d_+\iv) P_n 
$$
with $A:= T\iv(\psi) M(1,\phi) T\iv(\psi\iv)$.
We analyse the last three terms. First, using $P_n=W_n^2$, $Q_n=V_{n}V_{-n}$, we see that
$$
 P_n A Q_n T(d_+\iv) P_n = W_n \Big(W_n A V_n\Big)  H(d_+\iv) W_n 
 $$
 tends to zero in trace norm because $W_n A V_n\to 0$ strongly (see Proposition \ref{p2.5}).
 Secondly, 
 $$
P_n T(cd_+) Q_n A P_n = W_n H(\tc \td_+) \Big( V_{-n} A W_n\Big) W_n
$$
tends also to zero in trace norm because $(V_{-n} A W_n)^*\to 0$ strongly (again by Prop.~\ref{p2.5}). 
This implies that the last two terms of the above expressions tend to zero in trace norm.
Finally,
$$
P_n T(cd_+) Q_n A Q_n T(d_+\iv) P_n 
= W_n H(\tc\td_+) \Big(V_{-n} A V_n \Big) H(d_+\iv) W_n
$$
Here $V_{-n} A V_n\to I$ strongly (by Prop.~\ref{p2.5}). Hence the latter is $W_n H(\tc\td_+)  H(d_+\iv) W_n$ plus a sequence tending to zero in trace norm.

Summarizing, we have so far 
\begin{align*}
B_n &:=P_n T\iv(\psi) M(c,cd\phi) T\iv(\psi\iv) P_n
\\
&
= T_n(cd_+) P_n A P_n T_n(d_+\iv) + P_n K_1 P_n + W_n L_1 W_n + D_n^{(1)}
\end{align*}
with $K_1$ and $L_1=H(\tc\td_+)  H(d_+\iv)$ being trace class operators and $D_n^{(1)}$ being a sequence tending to zero in trace norm.
Now we take the inverses of $T_n(cd_+)$, $P_n A P_n=:A_n$, and $T_n(d_+\iv)$. 
Thus,
$$
T_n\iv(cd_+) B_n T_n(d_+\iv) A_n\iv =
$$
$$ 
P_n + T_n\iv(cd_+) P_n K_1 P_n T_n(d_+\iv) A_n\iv+
T_n\iv(cd_+) W_n L_1 W_n T_n(d_+\iv) A_n\iv+ D_n^{(2)}
$$
with $D_n^{(2)}\to 0$ in trace norm due to stability. Using stability and the strong convergence of the above sequences and their 
adjoints it follows that
$$
T_n\iv(cd_+) P_n K_1 P_n T_n(d_+\iv) A_n\iv= P_n K P_n +  D_n^{(3)}
$$
and 
$$
T_n\iv(cd_+) W_n L_1 W_n T_n(d_+\iv) A_n\iv= W_nL W_n  +  D_n^{(4)}
$$
with $D_n^{(j)}\to 0$ in trace norm and $K,L$ being trace class. In the latter we use that the 
strong limits of 
$$
W_n T_n(cd_+) W_n , \quad W_n A W_n=W_nA_nW_n,\quad W_nT_n(d_+\iv)W_n
$$
(and their adjoints) exist. Indeed, apply Proposition \ref{p2.5}. Also, due to stability  their inverses have a strong limit. 

Thus
\begin{equation}\label{f.eqn}
T_n\iv(cd_+) B_n T_n\iv(d_+\iv) A_n\iv = P_n + P_nK P_n +W_n L W_n +D_n
\end{equation}
with $D_n\to 0$ in trace norm. Write
$$
P_n+P_nK P_n+W_nLW_n=(P_n+P_nKP_n)(P_n+W_n L W_n) - P_n K W_n L W_n
$$
with the last term tending to zero in trace norm as $W_n\to0$ weakly.
Now take determinants and it follows that 
$$
\lim_{n\to\iy}
\frac{\det B_n}{\det T_n(cd_+) \cdot \det A_n \cdot \det T_n(d_+\iv ) }
=
\det( I+K)\cdot \det (I+L).
$$
From the standard Szeg\"o Limit theorem we get 
$$
\det T_n(cd_+) \sim G[cd_+]^n \cdot  \det  T(cd_+)  T(c\iv d_+\iv),\qquad n\to\iy,
$$
while $\det T_n(d_+\iv)= G[d_+\iv]^n$. Together we get the exponential factor $G[c]=G[cd_+]\cdot G[d_+\iv]$.

It remains to identify trace class operators $K$ and $L$. This is most conveniently done by passing to strong limits
(and the strong limits after applying $W_n$ from both sides) in \eqref{f.eqn}.
We obtain
$$
T\iv(cd_+) B T\iv(d_+\iv) A\iv = I+K
$$
with $B=T\iv(\psi) M(c,cd\phi) T\iv(\psi\iv)$ and $A=T\iv(\psi) M(1,\phi) T\iv(\psi\iv)$, i.e., 
$$
I+K= T\iv(cd_+)T\iv(\psi) M(c,cd\phi) T\iv(\psi\iv) T(d_+) T(\psi\iv) M(1,\phi)\iv T(\psi).
$$
This gives one of the operator determinant in \eqref{E.const}. As for the $W_n$-limits we obtain
$$
T\iv (\tc \td_+) T(\tc) T\iv(\td_+\iv) = I+L.
$$
Here notice that $W_nB_nW_n\to T(\tc)$ and $W_nA_nW_n\to I$ (again by Proposition \ref{p2.5}).
Thus, 
$$
\det (I+L)= \det T\iv(\tc \td_+) T(\tc) T(\td_+),
$$
which is the remaining term in  \eqref{E.const} along with the constant term from the Szeg\"o-Limit theorem above.
\end{proof}

In order to use the previous theorem we have to know the stability of the sequence $A_n$ defined in \eqref{f.An}.
This is a non-trivial issue and is addressed in \cite{BEx}, where the following two theorems are proved.
These results include certain ``local'' operators, which we are not going to define here, but instead refer to  \cite{BEx}.

\begin{theorem}\label{t5.2}
Let $\phi$ and $\psi$ be of the form \eqref{f.psi} and \eqref{f.phi}. Assume that conditions 
 \eqref{f.16} 
 are satisfied. Then the sequence 
$$
A_n=P_{n}\,T^{-1}(\psi) M(1, \phi ) T\iv(\psi^{-1})P_n
$$ 
is stable if and only if the following conditions are satisfied
\begin{itemize}
\item[(i)]
the operator $M(1,\phi)$ is invertible on $\ell^2$,
\item[(ii)]
$\Re \gamma^+\notin 2\Z+1/2$ and $\Re \gamma^-\notin 2\Z -1/2$,
\item[(iii)] 
for each $1\le r\le R$, a certain the ``local'' operator $B(\alpha_r^+,\alpha_r^-,\gamma_r)$ is invertible.
\end{itemize}
\end{theorem}

This theorem is proved in \cite{BEx} by using general stability results of \cite{R}. These general stability results imply that $A_n$ is stable if and only if
a certain collection of operators is invertible. Among these operators is the strong limit of $A_n$, i.e., the operator $A=T^{-1}(\psi) M(1, \phi ) T\iv(\psi^{-1})$. Thus it is necessary for stability that $M(1,\phi)$ is invertible. In addition, there occur ``local'' operators (associated to each point where $\psi$ or $\phi$ have jump discontinuities).
Invertibility of the local operators at $t=\pm1$ lead to conditions (ii). For the jumps at $t=\tau_r$
and $t=\bt_r$ the local operators are Mellin convolution operators $B(\alpha_r^+,\alpha_r^-,\gamma_r)$ with $2\times 2$ matrix valued symbol
defined in terms of the three parameters $\alpha_r^+,\alpha_r^-,\gamma_r\in \C$. As well known, the invertibility of such operators is equivalent to
a Wiener-Hopf factorization of the matrix symbol, which in general is an unaccessible problem.
Therefore, we have only the following results availabe \cite{BEx}.
On the positive side, part (c) covers the special cases we are  particularly interested in. 

\begin{theorem}\label{t5.3}
Let $\Re \alpha_r^\pm \notin \Z+1/2$.
\item[(a)]
If $B(\alpha_r^+,\alpha_r^-,\gamma_r)$ is invertible, then $\Re\gamma_r\notin \Z+1/2$;
\item[(b)]
If $\Re\gamma_r\notin \Z+1/2$, then $B(\alpha_r^+,\alpha_r^-,\gamma_r)$ is Fredholm with index zero;
\item[(c)]
If $\alpha_r=\alpha_r^+=\alpha_r^-$ and $\Re\gamma_r\notin \Z+1/2$, then $B(\alpha_r,\alpha_r,\gamma_r)$ is invertible.
\end{theorem}


Let us return to the invertibility of the operator $M(1,\phi)$. For general Toeplitz+Hankel operators (with jump discontinuities)
invertibility is a delicate issue. In \cite{BE5} necessary and sufficient condtions for invertibility conditions were established for Toeplitz+Hankel operator 
$T(a)+H(b)$ with piecewise continuous $a,b$ satisfying the additional condition $a\ta=b\tb$. Since 
 $\phi\tph=1$ our operator $M(1,\phi)=I+H(\phi)$ falls into this class and we cite the 
corresponding result (Corollary 5.5 of \cite{BE5}).
For  sake of  simple presentation we only state it in the form that provides us a sufficient condition.
Note that $\wt{u_{\tau,\gamma}}=u_{\bt,-\gamma}$\,, which ensures $\phi\tph=1$.

\begin{theorem}
Let $\phi$ be of the form 
$$
\phi=  u_{1,\gamma^+} \; u_{-1,\gamma^-}\;
\prod_{r=1}^R u_{\tau_r,\gamma_r}\;   u_{\bar{\tau}_r,\gamma_r}\;,
$$
with distinct $\tau_1,\dots,\tau_R\in \T_+$ and assume 
$$
-3/2<\Re \gamma^+ <1/2,\quad 
-1/2<\Re \gamma^- <3/2,\quad 
-1/2<\Re \gamma_r<1/2.
$$
Then $M(1,\phi)=I+H(\phi)$ is invertible on $\ell^2$.
\end{theorem}

We will make further remarks on the invertibility of $M(1,\phi)$ in Section \ref{sec:9}.

As a conclusion of the previous three results we can give some sufficient condition for stability.
Notice that if (b) in Theorem \ref{t5.3} would imply invertibility (as is the case  in (c)), we would not need the extra condition $\al_r^+=\al_r^-$.

\begin{corollary}
Let $\phi$ and $\psi$ be of the form \eqref{f.psi} and \eqref{f.phi}. Assume that conditions 
 \eqref{f.16}  and \eqref{f.17} are satisfied and that in addition $\alpha_r=\alpha_r^+=\alpha_r^-$.
 Then $A_n$ is stable.
\end{corollary}


\section{Determinant computations}

In view of our separation theorem, we need to do two things. One is to evaluate the constant \eqref{E.const} and the other is to compute the asymptotics of the determinant  of 
$A_{n}.$ The goal of this section is to do the first, that is, evaluate the constant \eqref{E.const} which is given in terms of operator determinants.
Some of the factors  have been computed before, and the complicated one can be reduced to simpler ones, which also have been computed before.

We start with the following definitions and observations. 
For $A,B\in L^\iy(\T)$ for which $H(A)H(\tilde{B})$ is trace class and for which $T(A)$ and $T(B)$ are invertible, let
$$
E[A,B]= \det \Big( T\iv(A) T(AB) T\iv(B)\Big).
$$
Note that $H(A)H(\tilde{B})$ is trace class if one of the functions is in $B^1_1$. If both functions are smooth (and have a continuous logarithm), then 
it has been shown that 
$$
E[A,B] = \exp\Big(\tr\, H(\log A) H(\log \tilde{B})\Big) =\exp\Big(\sum_{k\ge 1}k  [\log A]_{k} [\log B]_{-k}\Big).
$$
{}From this formula it follows that 
$$
E[A,B]=E[\tilde{B},\tilde{A}]=E[A_+,B_-],
$$
where $A=A_-A_+$ and $B=B_-B_+$ are Wiener-Hopf factorization of functions in $B_1^1$.

This constant is related to the constant 
$$
E[C]= \det T(C) T(C\iv) = \exp\Big(\sum_{k\ge 1} k [\log C]_k [\log C]_{-k}\Big).
$$
appearing in the Szeg\"o-Widom limit theorem. In fact, we have
$E[C]=E[C,C]=E[C_+,C_-]$. 

Finally, for  a nonvanishing function $C\in B^{1}_1$ with winding number zero let
$$
F[C]=\det( I + T\iv(C)H(C))
$$
It was computed in \cite{BE4} that
\begin{align}
F[C] &=\exp\Big(-\frac{1}{2} \tr\, H(\log C)^2+ \tr\, H(\log C)\Big)
\nonumber
\\
&=\exp \Big(-\frac{1}{2} \sum_{k\ge 1} k [\log C]_k^2+\sum_{k\ge1} [\log C]_{2k-1}\Big).
\label{F[c]}
\end{align}
The previous determinant relates to a slighty more complicated determinant.
\begin{lemma}
Let $C\in B^1_1$ be nonvanishing and have winding number zero. 
Assume that $\phi$ has a bounded Wiener-Hopf factorization $\phi=\phi_- \tilde{\phi}_-\iv$. Then 
$$
\det(I + H(C) T\iv(C\phi))=\det( I+ H(C) T\iv(C))=F[C].
$$

\end{lemma}\label{l6.1}
\proof
We have
$$
T(C\phi)= T(\phi_-)T(C) T(\tilde{\phi}_-\iv)
$$
with the factors being bounded invertible operators. Using
$$
H(C)T(\tph_-)=H(C\phi_-)=T(\phi_-)H(C)
$$
we obtain
$$
H(C)T\iv(C\phi)= H(C) T(\tilde{\phi}_-) T\iv(C) T(\phi_-\iv)= T(\phi_-) H(C) T\iv(C) T(\phi_-\iv),
$$
which gives the assertion.
\qed

This next lemma illustrates some properties of the constant $E[A,B]$ which will be used later to simplify determinants. 

\begin{lemma}\label{lemma.det} 
Let $A, B, C \in L^\iy(\T)$ be such that $T(A)$ is invertible and $B$ and $C$ admit bounded factorizations.
\begin{enumerate}
\item[(a)]
If  $ H(A)H(\tilde{B})$ and $H(A)H(\tilde{C})$ are trace class, then $$E[A,BC] = E[A,B] \cdot E[A,C].$$
\item[(b)]
If  $ H(B)H(\tilde{A})$ and $H(C)H(\tilde{A})$ are trace class, then
$$E[BC,A] = E[B,A] \cdot E[C,A].$$
\end{enumerate}
\end{lemma}

\proof
Let $B=B_-B_+$ and $C=C_-C_+$ be the bounded factorizations. 
Using \eqref{T1}--\eqref{T2} in what follows, we remark that 
$$H(A)H(\tilde{B}_-)=H(A)H(\tilde{B})T(B_+\iv)
\quad \mbox{ and }\quad 
H(A)H(\tilde{C}_-)=H(A)H(\tilde{C})T(C_+\iv)
$$ are trace class.
Analogously, $H(A)H(\tilde{B}\tilde{C}) = H(A)H(\tilde{B}_-\tilde{C}_-)T(B_+C_+)$.
Observe that 
\begin{align*}
H(A)H(\tilde{B}_-\tilde{C}_-) &= H(A)H(\tilde{B}_-)T(C_-) + H(A)T(\tilde{B}_-)H(\tilde{C}_-)
\\
&= H(A)H(\tilde{B}_-)T(C_-) + T(B_-)H(A) H(\tilde{C}_-)
\end{align*}
where we used $T(B_-)H(A)=H(B_-A)=H(AB_-)=H(A)T(\tilde{B}_-)$.
Therefore we conclude that $H(A)H(\tilde{B}\tilde{C})$ is trace class, too
The Toeplitz operators $T(B)$, $T(C)$, and $T(BC)$ are invertible due to the bounded factorizations. This implies that 
the three operator determinants are well-defined. A straightforward compuation using the factorizations yields that 
$$
E[A, B]=\det T\iv (A) T(AB) T\iv(B)=\det T\iv(A) T(AB_-) T(B_-\iv) = E[A,B_-]
$$ 
and similar statements for the other two determinants. In fact, we can write
$$
E[A,B]= \det T(B_-\iv) T\iv(A) T(AB_-),\qquad E[A,C]=\det T\iv(A) T(AC_-) T(C_-\iv)
$$
and multiplication yields
$$
\det T(B_-\iv) T\iv(A) T(B_-)T(A) T\iv(A) T(AC_-) T(C_-\iv)
$$
$$
= \det T(B_-\iv) T\iv(A) T(B_-) T(AC_-) T(C_-\iv)
$$
$$
=\det T\iv(A) T(AB_-C_-)T(C_-\iv B_-\iv),
$$
which is $E[A,BC]$. This proof (a). The proof of (b) is analogous.
%
%
\qed

The next theorem states that the operator determinant occuring in \eqref{E.const} is well defined under certain conditions (invertibility of $M(1,\phi)$),
and that it can be expressed in terms of above constants in case of a slightly stronger condition (invertibility of $T(\phi)$). Afterwards, when we specialize to 
the functions $\psi$ and $\phi$ with jump discontinuities we will see that the stronger condition is redundant for the evaluation of the constant.
Notice that we have already shown in Theorem \ref{t5.1} that operator determinant is well-defined, but under the (perhaps stronger) assumption of the stability of a certain sequence (which in fact implies invertibility of $M(1,\phi)$).

\begin{theorem}\label{t6.2}
Let $\psi,\phi\in L^\iy(\T)$ with $\phi\tilde{\phi}=1$ such that $T(\psi)$  is invertible on $\ell^2$.
Assume that $c\in B^1_1$ is nonvanishing and has winding number zero and that $d\in B^1_1$ has a factorization $d=d_+\td_+\iv$
in $B_1^1$.
\begin{itemize}
\item[(a)]
If $M(1,\phi)$ is invertible on $\ell^2$, then the following operator determinant is well-defined:
$$
E_1=\det \Big( T\iv(c d_+) T\iv(\psi) M(c,cd\phi) T\iv(\psi\iv) T(d_+) T(\psi\iv) M\iv(1,\phi)T(\psi)\Big).
$$
\item[(b)]
If $T(\phi)$ is invertible on $\ell^2$, then $M(1,\phi)$ is invertible, and
$$
E_1=
\frac{E[\psi,cd_+]}{E[cd_+,\psi]}\times E[d_+\iv ,\psi\iv] \times E[cd_+,\phi]\times 
\det\Big((T(cd_+\phi)+H(cd_+)) T\iv(cd_+\phi)\Big)\,.
$$
\item[(c)]
If $\phi$ and $\psi$ have a bounded factorization, then
$$
E_1=
\frac{E[\psi,c]}{E[c,\psi]}\times E[cd_+,\phi]\times 
F[cd_+]\,.
$$
\end{itemize}
\end{theorem}
\begin{proof}
(a):\
Abbreviate $e=cd_+$ and write 
$$
M(e, e\phi)=T(e)+H(e\phi)= T(e)+T(e)H(\phi)+H(e)T(\tilde{\phi})=T(e)M(1,\phi)+K_1
$$
with $K_1$ being trace class. This implies 
$$
\det M(cd_+,cd_+\phi)M\iv(1,\phi) T\iv(cd_+)=
\det  M\iv(1,\phi) T\iv(cd_+) M(cd_+,cd_+\phi)
$$
is well-defined. Multiplying this determinant with the well-defined determinant 
\begin{align*}
E[d_+\iv,\psi\iv] 
&=\det T\iv(d_+\iv)T(\psi\iv d_+\iv)T\iv(\psi\iv)
\\
&=\det T(d_+\iv) T\iv(\psi\iv) T(d_+) T(\psi\iv)
\end{align*}
and  observing 
$M(cd_+,cd_+\phi) T(d_+\iv)= M(c,cd\phi)$ yields the well-defined determinant
$$
\det  M\iv(1,\phi) T\iv(cd_+) M(c,cd\phi) T\iv(\psi\iv) T(d_+) T(\psi\iv),
$$
which we can also write as
$$
\det  M(c,cd\phi) T\iv(\psi\iv) T(d_+) T(\psi\iv)  M\iv(1,\phi) T\iv(cd_+).
$$

Next observe that 
\begin{align*}
\frac{E[\psi,cd_+]}{E[cd_+,\psi]}
&=\det T(cd_+)T(\psi) T\iv(\psi cd_+)\cdot \det T(\psi cd_+)T\iv( cd_+) T\iv (\psi).
\\
&= \det T(cd_+) T(\psi)T\iv( cd_+) T\iv (\psi).
\end{align*}
This well-defined determinant we multiply to the above one to obtain
$$
\det  M(c,cd\phi) T\iv(\psi\iv) T(d_+) T(\psi\iv)  M\iv(1,\phi) T(\psi)T\iv( cd_+) T\iv (\psi),
$$
which is $E_1$. Summarizing, besides the issue of $E_1$ being well-defined we have shown that 
$$
E_1= \frac{E[\psi,cd_+]}{E[cd_+,\psi]}\times E[d_+\iv,\psi\iv] \times 
\det M(cd_+, cd_+ \phi) M\iv(1,\phi) T\iv(c d_+).
$$

(b) Now we are going to show that $M(1,\phi)$ is invertible if so is $T(\phi)$ and express the inverse.
Then we will compute the remaining determinant.
The identity 
$$
(I+H(\phi))\,(I-H(\phi))=I-H(\phi)H(\tilde{\phi}\iv) = T(\phi)\,T(\phi\iv)
$$
implies that $M(1,\phi)=I+H(\phi)$ is invertible. Moreover,
$$
M\iv(1,\phi)=
(I+H(\phi))\iv = (I-H(\phi))\;T\iv(\phi\iv)\, T\iv (\phi).
$$
Next observe that
$$
M(cd_+,cd_+\phi)(I-H(\phi))=T(cd_+)+H(cd_+\phi) - T(cd_+)H(\phi)-H(cd_+\phi)H(\phi)
$$
$$
=T(cd_+)+H(cd_+\phi) - H(cd_+\phi)+H(cd_+)T(\tilde{\phi})-T(cd_+\phi\tilde{\phi})+T(cd_+\phi)T(\tilde{\phi})
$$
$$
=H(cd_+)T(\tilde{\phi})+T(cd_+\phi)T(\tilde{\phi})
$$
$$
=(T(cd_+\phi)+H(cd_+))\; T(\phi\iv).
$$
Therefore,
$$
M(cd_+,cd_+\phi)M\iv(1,\phi)=M(cd_+,cd_+\phi)(I-H(\phi)) T\iv (\phi\iv)\, T\iv (\phi)
$$
$$
=
(T(cd_+\phi)+H(cd_+))T\iv(\phi),
$$
and thus
$$
\det\Big(
M(cd_+, cd_+ \phi) M\iv(1,\phi) T\iv(c d_+)\Big)
=\det\Big((T(cd_+\phi)+H(cd_+))T\iv(\phi) T\iv(cd_+)\Big).
$$
We split this into
$$
\det\Big((T(cd_+\phi)+H(cd_+)) T\iv(cd_+\phi)\Big)\times
\det\Big( T(cd_+\phi) T\iv(\phi) T\iv(cd_+)\Big)
$$
with the last determinant equal to $E[cd_+,\phi]$.
For part (c), we apply Lemma \ref{l6.1} and Lemma \ref{lemma.det}.
\end{proof}

Now we specialize to the functions we are interested in.

\begin{corollary}
Let $\psi$ and $\phi$ be given by \eqref{f.psi} and \eqref{f.phi} and assume that the condition \eqref{f.16} and \eqref{f.17} hold. 
Moreover, let $c=c_-G[c]c_+$ and $d=\td_+\iv d_+$ be factorizations in $B_1^1$ (see \eqref{f.c+-}--\eqref{f.dd+}). 
Then the operator determinant $E_1$ of Theorem \ref{t6.2} is well-defined and given by
\begin{align*}
\hat{E}_1=\frac{E_1}{F[e_+]}&=
c_0(1)^{-\alpha^+} c_0(-1)^{-\alpha^-}\prod_{r=1}^R c_0(\tau_r)^{-\alpha_r^+}c_0(\bt_r)^{-\alpha^-_r}
\\
&\quad \times e_+(1)^{\alpha^++\beta^+}e_+(-1)^{\alpha^-+\beta^-}
\prod_{r=1}^R
e_+(\tau_r)^{\alpha_r^++\beta_r}e_+(\bt_r)^{\alpha^-_r+\beta_r} \;
\end{align*}
with $e_+=c_+d_+$ and $c_0= c_+ c_-=c/G[c]$.
\end{corollary}
\begin{proof}
The assumptions imply that $T(\psi)$ and $M(1,\phi)$ are invertible. Hence by the previous theorem, part (a), 
the operator determinant $E_1$ is well-defined. It also depends analytically on the parameters
$\alpha^\pm, \beta^\pm, \alpha_r^\pm, \beta_r$.
Therefore it is sufficient to prove the identity under the assumption that the real parts of all parameters vanish. This implies that 
$T(\phi)$ is invertible and hence part (c) of the previous theorem can be applied.  
Moreover, $F[cd_+]=F[c_+d_+]$ as can be easily seen from \eqref{F[c]}.

To carry out the computation of the various $E[\cdot,\cdot]$ terms, we observe that for functions $C$ admitting a factorization in $B^1_1$ (see \eqref{f.c+-}) 
the following general formulas were established in  \cite[Sect.~10.62]{BS}),
$$
E[u_{\tau,\beta},C]=C_-(\tau)^{-\beta},\qquad
E[C,u_{\tau,\beta}]=C_+(\tau)^{\beta}.
$$
Recalling the definition of $\psi$ and $\phi$, 
\begin{align*}
\psi &= u_{1,\alpha^+}\;   u_{-1,\alpha^-}
\prod_{r=1}^R u_{\tau_r,\alpha^+_r}\;  u_{\bar{\tau}_r,\alpha^-_r}\;,
\\
\phi  &=   u_{1,\alpha^++\beta^+} \; u_{-1,\alpha^-+\beta^-}\;
\prod_{r=1}^R u_{\tau_r,\alpha_r+\beta_r}\;   u_{\bar{\tau}_r,\alpha_r+\beta_r}\;,
\end{align*}
we get
\begin{align*}
E[\psi,c_-] &= c_-(1)^{-\alpha^+} c_-(-1)^{-\alpha^-}\prod_{r=1}^R c_-(\tau_r)^{-\alpha_r^+}c_-(\bt)^{-\alpha^-_r},
\\
E[c_+,\psi]\iv &= c_+(1)^{-\alpha^+} c_+(-1)^{-\alpha^-}\prod_{r=1}^R c_+(\tau_r)^{-\alpha_r^+}c_+(\bt)^{-\alpha^-_r},
\\
E[c_+d_+,\phi] &= e_+(1)^{\alpha^++\beta^+}e_+(-1)^{\alpha^-+\beta^-}
\prod_{r=1}^R
e_+(\tau_r)^{\alpha_r^++\beta_r}e_+(\bt)^{\alpha^-_r+\beta_r},
\end{align*}
from which the formula follows.
\end{proof}

We remark that in the special case of $\alpha_r=\alpha_r^\pm$ and $\beta_r=0$, which we are going to consider later, the constant
\begin{align}
\hat{E}_1 &= c_+(1)^{\beta^+}c_+(-1)^{\beta^-} c_-(1)^{-\alpha^+} c_-(-1)^{-\alpha^-}
\prod_{r=1}^R c_-(\tau_r)^{-\alpha_r} c_-(\bt_r)^{-\alpha_r}
\nonumber
\\
&\quad \times
d_+(1)^{\alpha^++\beta^+} d_+(-1)^{\alpha^-+\beta^-}
\prod_{r=1}^R  d_+(\tau_r)^{\alpha_r} d_+(\bt_r)^{\alpha_r}\,.
\label{f.hE1}
\end{align}

Let us now turn to the constant $E$ appearing in Theorem \ref{t5.1}.

\begin{corollary}\label{c6.5}
Let $\psi$ and $\phi$ be given by \eqref{f.psi} and \eqref{f.phi} and assume that the condition \eqref{f.16} and \eqref{f.17} hold. 
Moreover, let $c=c_-G[c]c_+$ and $d=\td_+\iv d_+$ be factorizations in $B_1^1$ (see \eqref{f.c+-}--\eqref{f.dd+}). 
Then the constant $E$ in \eqref{E.const} is well-defined and given by
\begin{eqnarray*}
E= \hat{E}_1\times \left(\frac{c_+(1)d_+(1)}{c_+(-1)d_+(-1)}\right)^{1/2}\times \exp\Big (\sum_{k\ge 1} k [\log c]_k[\log c]_{-k} -\frac{1}{2}\sum_{k\ge 1} k ([\log c]_k+[\log d]_k)^2\Big).
\end{eqnarray*}
where $\hat{E}_1$ is the expression in the previous corollary.
\end{corollary}
\begin{proof}
We have to identify the additional constants,
$$
\det\Big( T\iv(\tc \td_+) T(\tc) T(\td_+)\Big) \times \det \Big( T(cd_+)T(c\iv d_+\iv)\Big)
$$
which are 
$$
E^{-1}[\tilde{c},\tilde{d}_+]\cdot E[cd_+]= \frac{E[cd_+,cd_+]}{E[d_+,c]}=E[c_+,c_-].
$$
This we combine with $F[cd_+]=F[e_+]$ and the previous corollary.
\end{proof}

Putting all this together, we have the following.

\begin{corollary}\label{c6.4}
Let $\psi$ and $\phi$ be given by \eqref{f.psi} and \eqref{f.phi} and assume that the condition \eqref{f.16} and \eqref{f.17} hold. 
Moreover, let $c=c_-G[c]c_+$ and $d=\td_+\iv d_+$ be factorizations in $B_1^1$ (see \eqref{f.c+-}--\eqref{f.dd+}). 
Finally, suppose that (iii) of Theorem \ref{t5.2} holds. Then
$$
\lim_{n\to\infty}
\frac{
\det\Big( P_{n} T\iv(\psi) M(c, c d \phi )  T\iv(\psi^{-1}) P_{n}\Big)}{
G[c]^n  \times  \det \Big( P_{n}\,T^{-1}(\psi) M(1, \phi ) T\iv(\psi^{-1})P_n\Big) }
 = E,
 $$
where $E$ is as in the previous corollary.
\end{corollary}

\section{Known asymptotics}

In the previous separation theorem and the constant computation we have reduced the asymptotics of
$$
\det\Big( P_{n} T\iv(\psi) M(c, c d \phi )  T\iv(\psi^{-1}) P_{n}\Big)
$$
to the asymptotics of 
$$
\det \Big( P_{n}\,T^{-1}(\psi) M(1, \phi ) T\iv(\psi^{-1})P_n\Big) .
$$
The separation theorem is of course only as useful as  far as we are able to obtain this last asymptotics. We can reverse the considerations of Section 
\ref{s3} and Theorem \ref{t4.1} and obtain 
$$
\det \Big( P_{n}\,T^{-1}(\psi) M(1, \phi ) T\iv(\psi^{-1})P_n\Big) = \det M_n(a_0, b_0)
$$
where 
\begin{align}
a_0 &= v_{1,\alpha^+}\;   v_{-1,\alpha^-}\,
\prod_{r=1}^R v_{\tau_r,\alpha^+_r}\;  v_{\bar{\tau}_r,\alpha^-_r}\;,
\\
b_0 &= v_{1,\alpha^+}\;  u_{1,\beta^+} \; v_{-1,\alpha^-}\; u_{-1,\beta^-}\;
\prod_{r=1}^R v_{\tau_r,\alpha_r}u_{\tau_r,\beta_r}\;  v_{\bar{\tau}_r,\alpha_r} u_{\bar{\tau}_r,\beta_r}\,.
\end{align}
These are the original functions $a$ and $b$ without the $c$ and $d$ terms.

The asymptotics of $\det M_n(a_0,b_0)$ are known in the cases of $\beta_r=0$ and $\alpha_r=\alpha_r^+=\alpha_r^-$ and $\beta^+\in\{0,-1\}$, $\beta^-\in\{0,1\}$.
We remark that in these cases $a_0$ is an even function which is a product of pure Fisher-Hartwig type functions with zeros/poles only (no jumps).
Furthermore, depending on the values of $\beta^\pm$, we have four cases,
$$
b_0(t)=\pm a_0(t),\qquad b_0(t)= t a_0(t),\qquad b_0(t)=-t^{-1} a_0(t),
$$
which are precisely the cases (i)--(iv) described in the introduction. We list them here and note that $G(1 +z)$ is the Barnes $G$-function, an entire function satisfying 
$G( 1 + z) = \Gamma(z) G(z)$ (see \cite{Bar}).

The asymptotics of 
$\det M_n(a_0, b_0)$ are given by

\begin{enumerate}
\item [(1)] $b_0(t)= a_0(t),\,\, \beta_r=0$ and $\alpha_r=\alpha_r^+=\alpha_r^-$, and $\beta^+ = 0$, $\beta^- = 0.$ 
\[  n^{\{\frac{1}{2} ((\al^{+})^{2} + (\al^{-})^{2} - \al^{+}+ \al^{-})  + \sum \alpha_{r}^{2} \}}\,\,2^{-\frac{1}{2} (\al^{+} + \al^{-})^{2} +\frac{1}{2}(\al^{+} + \al^{-}) +\sum \al_{r}^{2}}\]
\[ \times  \prod _{r} | 1 - \tau_{r}^{2}|^{-\al_{r}^{2}} | 1 - \tau_{r}|^{-2\al_{r}(\al^{+} -1/2)} | 1 +\tau_{r}|^{-2\al_{r} (\al^{-} +1/2)}\]
 \[ \times \prod_{ j <k} | \tau_{k} - \tau_{j} |^{-2 \al_{k}\al_{j}} \, | \tau_{k} - 1/\tau_{j} |^{-2 \al_{k}\al_{j}}\]
 \[ \times \frac{ \pi ^{\frac{1}{2} (\al^{+} + \al ^{-})} G(1/2)G(3/2)}{G(1/2 + \al^{+})G( 3/2 + \al^{-})}
\prod_{r} \frac{ G( 1 + \al_{r})^{2}}{G( 1 + 2 \al_{r})}\]

\item [(2)] $b_0(t)= -a_0(t),\,\,\beta_r=0$ and $\alpha_r=\alpha_r^+=\alpha_r^-$, and $\beta^+ = -1$, $\beta^- = 1.$ 

\[  n^{\{\frac{1}{2} ((\al^{+})^{2} + (\al^{-})^{2} + \al^{+}- \al^{-})  + \sum \alpha_{r}^{2} \}}\,
 2^{-\frac{1}{2} (\al^{+} + \al^{-})^{2} +\frac{1}{2}(\al^{+} + \al^{-}) +\sum \al_{r}^{2}}\]
 \[ \times \prod _{r} | 1 - \tau_{r}^{2}|^{-\al_{r}^{2}} | 1 - \tau_{r}|^{-2\al_{r} (\al^{+} +1/2)} | 1 +\tau_{r}|^{-2\al_{r} (\al^{-} -1/2)}\]
 \[\times  \prod_{ j <k} | \tau_{k} - \tau_{j} |^{-2 \al_{k}\al_{j}} \, | \tau_{k} - 1/\tau_{j} |^{-2 \al_{k}\al_{j}}\]
 \[ \times  \frac{ \pi ^{\frac{1}{2} (\al^{+} + \al ^{-})} G(3/2)G(1/2)}{G(3/2 + \al^{+})G( 1/2 + \al^{-})}
\prod_{r} \frac{ G( 1 + \al_{r})^{2}}{G( 1 + 2 \al_{r})}\]

\item[(3)] $b_0(t)= t a_0(t),\,\,\beta_r=0$ and $\alpha_r=\alpha_r^+=\alpha_r^-$, and $\beta^+ = 0$, $\beta^- = 1.$

\[  n^{\{\frac{1}{2} ((\al^{+})^{2} + (\al^{-})^{2} - \al^{+}- \al^{-})  + \sum \alpha_{r}^{2} \}}\,
 2^{\{2 -\frac{1}{2} (\al^{+} + \al^{-} -1)^{2} +\frac{1}{2}(\al^{+} + \al^{-} -1) +\sum \al_{r}^{2} \}}\]
 \[ \times \prod _{r} | 1 - \tau_{r}^{2}|^{-\al_{r}^{2}} | 1 - \tau_{r}|^{-2\al_{r} (\al^{+} -1/2)} | 1 +\tau_{r}|^{-2\al_{r}(\al^{-} -1/2)}\]
 \[\times  \prod_{ j <k} | \tau_{k} - \tau_{j} |^{-2 \al_{k}\al_{j}} \, | \tau_{k} - 1/\tau_{j} |^{-2 \al_{k}\al_{j}} \]
 \[ \times \frac{ \pi ^{\frac{1}{2} (\al^{+} + \al ^{-})} G(1/2)^{2}}{G(1/2 + \al^{+})G( 1/2 + \al^{-})}
\prod_{r} \frac{ G( 1 + \al_{r})^{2}}{G( 1 + 2 \al_{r})}\]

\item [(4)] $b_0(t)= -t^{-1}a_0(t),\,\,\beta_r=0$ and $\alpha_r=\alpha_r^+=\alpha_r^-$, and $\beta^+ = -1$, $\beta^- = 0.$

\[  n^{\{\frac{1}{2} ((\al^{+})^{2} + (\al^{-})^{2} + \al^{+} +\al^{-})  + \sum \alpha_{r}^{2} \}}\,
 2^{\{ -\frac{1}{2} (\al^{+} + \al^{-} +1)^{2} +\frac{1}{2}(\al^{+} + \al^{-} +1) +\sum \al_{r}^{2} \}}\]
 \[ \times \prod _{r} | 1 - \tau_{r}^{2}|^{-\al_{r}^{2}} | 1 - \tau_{r}|^{-2\al_{r} (\al^{+} +1/2)} | 1 +\tau_{r}|^{-2\al_{r}(\al^{-} +1/2)}\]
 \[\times  \prod_{ j <k} | \tau_{k} - \tau_{j} |^{-2 \al_{k}\al_{j}} \, | \tau_{k} - 1/\tau_{j} |^{-2 \al_{k}\al_{j}} \]
\[ \times \frac{ \pi ^{\frac{1}{2} (\al^{+} + \al ^{-})} G(3/2)^{2}}{G(3/2 + \al^{+})G( 3/2 + \al^{-})}
\prod_{r} \frac{ G( 1 + \al_{r})^{2}}{G( 1 + 2 \al_{r})}\]

\end{enumerate}

The above asymptotics have been proved in \cite[Theorem 1.25]{DIK}. In the special case of $\alpha^+=\alpha^-=0$ and case (1), see also \cite{BE2}.
On the other hand, if all $\alpha_r=0$, then the determinants  can be evaluated explicitely because they are related to Hankel determinants constructed from the moments of classical Jacobi orthogonal polynomials. In this later case the Hankel determinant asymptotics can also be described as Toeplitz determinants \cite{BE2}.


\section{Final computations for the four important cases}

Using the previous section and then our computations for the constant $E$ we have the final four answers. We assume in what follows that 
we have the factorizations $c=c_-G[c] c_+$ and $d=d_+\td_+\iv$  (see \eqref{f.c+-}--\eqref{f.dd+}) and that condition \eqref{f.16x} holds.
%
%
Then the asymptotics of $\det M_n(a, b)$ are given by the following products.

\begin{enumerate}

\item[(1)]  $b(t)= a(t),\,\,\beta_r=0$ and $\alpha_r=\alpha_r^+=\alpha_r^-$, and $\beta^+ = 0$, $\beta^- = 0.$

\[  G[c]^n\,n^{\{\frac{1}{2} ((\al^{+})^{2} + (\al^{-})^{2} - \al^{+}+ \al^{-})  + \sum \alpha_{r}^{2} \}}\, 2^{-\frac{1}{2} (\al^{+} + \al^{-})^{2} +\frac{1}{2}(\al^{+} + \al^{-}) +\sum \al_{r}^{2}}\]
\[   
\times \exp\left(\sum_{k\ge 1} \Big( k [\log c]_k[\log c]_{-k} -\frac{1}{2} k ([\log c]_k+[\log d]_k)^2\Big)\right)
\]
\[\times \, c_+(1)^{1/2} c_+(-1)^{-1/2} c_{-}(1)^{-\al^{+}} c_{-}(-1)^{-\al^{-}} \,\,\prod_{r} c_{-}(\tau_{r})^{-\al_{r}} c_{-}(\bar{\tau}_{r})^{-\al_{r}} \]
\[\times \, d_{+}(1)^{\al^{+}+1/2} d_{+}(-1)^{\al^{-}-1/2} \,\,\prod_{r} d_{+}(\tau_{r})^{\al_{r}} d_{+}(\bar{\tau}_{r})^{\al_{r}} \]
\[ \times \, \prod_{ j <k} | \tau_{k} - \tau_{j} |^{-2 \al_{k}\al_{j}} \, | \tau_{k} - 1/\tau_{j} |^{-2 \al_{k}\al_{j}} \]
 \[ \times \, \prod _{r} | 1 - \tau_{r}^{2}|^{-\al_{r}^{2}} | 1 - \tau_{r}|^{-2\al_{r} (\al^{+} -1/2)} | 1 +\tau_{r}|^{-2\al_{r} (\al^{-} +1/2)}\]
 \[ \times  \, \frac{ \pi ^{\frac{1}{2} (\al^{+} + \al ^{-})} G(1/2)G(3/2)}{G(1/2 + \al^{+})G( 3/2 + \al^{-})}
\prod_{r} \frac{ G( 1 + \al_{r})^{2}}{G( 1 + 2 \al_{r})}\]

\item[(2)] $b(t)= -a(t),\,\,\beta_r=0$ and $\alpha_r=\alpha_r^+=\alpha_r^-$, and $\beta^+ = -1$, $\beta^- = 1.$

\[  G[c]^n\,n^{\{\frac{1}{2} ((\al^{+})^{2} + (\al^{-})^{2} + \al^{+}- \al^{-})  + \sum \alpha_{r}^{2} \}}\,
 2^{-\frac{1}{2} (\al^{+} + \al^{-})^{2} +\frac{1}{2}(\al^{+} + \al^{-}) +\sum \al_{r}^{2}}\]
\[  
\times \exp\left(\sum_{k\ge 1} \Big( k [\log c]_k[\log c]_{-k} -\frac{1}{2} k ([\log c]_k+[\log d]_k)^2\Big)\right)
\]
 \[ \times \, c_+(1)^{-1/2} c_+(-1)^{1/2}c_{-}(1)^{-\al^{+}} c_{-}(-1)^{-\al^{-}} \,\,\prod_{r} c_{-}(\tau_{r})^{-\al_{r}} c_{-}(\bar{\tau}_{r})^{-\al_{r}} \]
 \[\times \, d_{+}(1)^{\al^{+}-1/2} d_{+}(-1)^{\al^{-}+1/2} \,\,\prod_{r} d_{+}(\tau_{r})^{\al_{r}} d_{+}(\bar{\tau}_{r})^{\al_{r}} \]
 \[ \times \prod _{r} | 1 - \tau_{r}^{2}|^{-\al_{r}^{2}} | 1 - \tau_{r}|^{-2\al_{r} (\al^{+} +1/2)} | 1 +\tau_{r}|^{-2\al_{r} (\al^{-} -1/2)}\]
 \[\times  \prod_{ j <k} | \tau_{k} - \tau_{j} |^{-2 \al_{k}\al_{j}} \, | \tau_{k} - 1/\tau_{j} |^{-2 \al_{k}\al_{j}} \]
 \[ \times \frac{ \pi ^{\frac{1}{2} (\al^{+} + \al ^{-})} G(3/2)G(1/2)}{G(3/2 + \al^{+})G( 1/2 + \al^{-})}
\prod_{r} \frac{ G( 1 + \al_{r})^{2}}{G( 1 + 2 \al_{r})}\]

\item[(3)] $b(t)= ta(t),\,\,\beta_r=0$ and $\alpha_r=\alpha_r^+=\alpha_r^-$, and $\beta^+ = 0$, $\beta^- = 1.$

\[  G[c]^n\,n^{\{\frac{1}{2} ((\al^{+})^{2} + (\al^{-})^{2} - \al^{+}- \al^{-})  + \sum \alpha_{r}^{2} \}}\,
 2^{\{2 -\frac{1}{2} (\al^{+} + \al^{-} -1)^{2} +\frac{1}{2}(\al^{+} + \al^{-} -1) +\sum \al_{r}^{2} \}}\]
 \[  
\times \exp\left(\sum_{k\ge 1} \Big( k [\log c]_k[\log c]_{-k} -\frac{1}{2} k ([\log c]_k+[\log d]_k)^2\Big)\right)
\] 
  \[ \times  c_{+}(1)^{1/2} c_{+}(-1)^{1/2}  \, c_{-}(1)^{-\al^{+}} c_{-}(-1)^{-\al^{-}} \,\,\prod_{r} c_{-}(\tau_{r})^{-\al_{r}} c_{-}(\bar{\tau}_{r})^{-\al_{r}} \]
  \[\times \, d_{+}(1)^{\al^{+}+1/2} d_{+}(-1)^{\al^{-}+1/2} \,\,\prod_{r} d_{+}(\tau_{r})^{\al_{r}} d_{+}(\bar{\tau}_{r})^{\al_{r}} \]
 \[ \times \prod _{r} | 1 - \tau_{r}^{2}|^{-\al_{r}^{2}} | 1 - \tau_{r}|^{-2\al_{r} (\al^{+} -1/2)} | 1 +\tau_{r}|^{-2\al_{r}(\al^{-} -1/2)}\]
 \[\times  \prod_{ j <k} | \tau_{k} - \tau_{j} |^{-2 \al_{k}\al_{j}} \, | \tau_{k} - 1/\tau_{j} |^{-2 \al_{k}\al_{j}} \]
 \[ \times \frac{ \pi ^{\frac{1}{2} (\al^{+} + \al ^{-})} G(1/2)^{2}}{G(1/2 + \al^{+})G( 1/2 + \al^{-})}
\prod_{r} \frac{ G( 1 + \al_{r})^{2}}{G( 1 + 2 \al_{r})}\]

\item [(4)] $b(t)= -t^{-1}a(t),\,\,\beta_r=0$ and $\alpha_r=\alpha_r^+=\alpha_r^-$, and $\beta^+ = -1$, $\beta^- = 0.$

\[ G[c]^n\, n^{\{\frac{1}{2} ((\al^{+})^{2} + (\al^{-})^{2} + \al^{+} +\al^{-})  + \sum \alpha_{r}^{2} \}}\,
 2^{\{ -\frac{1}{2} (\al^{+} + \al^{-} +1)^{2} +\frac{1}{2}(\al^{+} + \al^{-} +1) +\sum \al_{r}^{2} \}}\]
   \[  
\times \exp\left(\sum_{k\ge 1} \Big( k [\log c]_k[\log c]_{-k} -\frac{1}{2} k ([\log c]_k+[\log d]_k)^2\Big)\right)
\] 
  \[ \times c_{+}(1)^{-1/2} c_{+}(-1)^{-1/2}  \, c_{-}(1)^{-\al^{+}} c_{-}(-1)^{-\al^{-}} \,\,\prod_{r} c_{-}(\tau_{r})^{-\al_{r}} c_{-}(\bar{\tau}_{r})^{-\al_{r}} \]
  \[\times \, d_{+}(1)^{\alpha^+-1/2} d_{+}(-1)^{\al^--1/2} \,\,\prod_{r} d_{+}(\tau_{r})^{\al_{r}} d_{+}(\bar{\tau}_{r})^{\al_{r}} \]
 \[ \times \prod _{r} | 1 - \tau_{r}^{2}|^{-\al_{r}^{2}} | 1 - \tau_{r}|^{-2\al_{r} (\al^{+} +1/2)} | 1 +\tau_{r}|^{-2\al_{r}(\al^{-} +1/2)}\]
 \[\times  \prod_{ j <k} | \tau_{k} - \tau_{j} |^{-2 \al_{k}\al_{j}} \, | \tau_{k} - 1/\tau_{j} |^{-2 \al_{k}\al_{j}} \]
\[ \times \frac{ \pi ^{\frac{1}{2} (\al^{+} + \al ^{-})} G(3/2)^{2}}{G(3/2 + \al^{+})G( 3/2 + \al^{-})}
\prod_{r} \frac{ G( 1 + \al_{r})^{2}}{G( 1 + 2 \al_{r})}\]

\end{enumerate}

These formulas follow in a straightforward manner from simply using the known results quoted in the previous section, along with our separation theorem formulas for the constant and the evaluation of those for the specific values of $\beta^{+}$ and $\beta^{-}$. More specifically, use \eqref{f.hE1} and Corollary \ref{c6.5}.


\section{The inverse of $I + H(\psi)$}
\label{sec:9}

In the course of the computations in the previous sections, we discovered that when $\psi^{-1} = \tilde{\psi},$ explicit forms of the inverse of $I + H(\phi_0\psi)$ can be found in the four cases
$\phi_0=\pm1$, $\phi_0(z)=z$, $\phi_0(z)=-z\iv$ considered in the previously.  The inverse will be expressed in terms of the inverse of  $T(\psi)$ and $T(\psi\iv)$.

The first two cases are rather simple.
\begin{proposition} 
Suppose that $\psi^{-1} = \tilde{\psi}$ and that the operator $T(\psi)$ is invertible. Then $I\pm H(\psi)$ is invertible and
$$
(I\pm H(\psi))\iv = 
T(\psi^{-1})^{-1} (I\pm H(\psi\iv))T(\psi)^{-1}\,.
$$
\end{proposition}
\begin{proof}
From $\psi^{-1} = \tilde{\psi}$ and thus 
$ T(\psi)T(\psi^{-1}) = I - H(\psi)^{2}$,
we know that 
\[ (T(\psi)T(\psi^{-1}) )^{-1}= T(\psi^{-1})^{-1} T(\psi)^{-1} = (I - H(\psi)^{2})^{-1}.\]
This means that 
\[ T(\psi^{-1})^{-1} (I\pm H(\psi\iv)) T(\psi)^{-1} = (I - H(\psi)^{2})^{-1}T(\psi) \,(I\pm H(\psi\iv)) \,T(\psi)^{-1}. \]
Now from (\ref{H1}) we know that $T(\psi)H(\psi^{-1}) = - H(\psi) T(\psi).$ So the term
\[ T(\psi) \, (I \pm H( \psi^{-1}) \,T(\psi)^{-1} = I \mp H(  \psi) \]
and this yields the result.
\end{proof}

In order to continue with the other two case, we need to make a connection with factorization.
If $T(\psi)$ is invertible then there exists a Wiener-Hopf factorziation of $\psi$ 
such that $\psi = \psi_{-} \psi_{+} $ and $\psi_{\pm}$ and $\psi_{\pm}^{-1}$ are all  in $L^{2}$.
If in addition $\psi\iv=\tps$, then using the uniqueness of the Wiener-Hopf factorization (up to multiplicative constants)
it is not difficult to see that the factors are related to each other either by
$$
\psi_+=\tps_-\iv
\quad\mbox{ or }\quad
\psi_+=-\tps_-\iv\,.
$$
This means that we have either 
$$
\psi=\tps_+\iv \psi_+\quad \mbox{ or }\quad  \psi=-\tps_+\iv \psi_+\,.
$$
We can equivalently express this dichotomy as
$$
P_1T\iv(\psi) P_1=1
\quad\mbox{ or }\quad
P_1T\iv(\psi) P_1=-1
$$
by using that $T\iv(\psi)=T(\psi_+\iv) T(\psi_-\iv)$.
We need one basic result before we proceed.

\begin{lemma}\label{lemma.sum}
If $\psi^{-1} = \tilde{\psi}$, if the Toeplitz operator $T(\psi)$ is invertible and if $P_1T \iv(\psi) P_1=1$, then
$$ (H(\tilde{\psi}) + T(\psi)^{-1} H(\psi) - H(\tilde{\psi} /z) T(\psi)^{-1}H(\psi) )H(z) = 0.$$
\end{lemma}
\begin{proof}
If $T(\psi)$ is invertible then there exists a factorization of $\psi$ such that 
$\psi = \psi_{-} \psi_{+} $ where $\psi_{-} = \tilde{\psi_{+}}^{-1}$ and $\psi_{\pm}$ and $\psi_{\pm}^{-1}$ are all  in $L^{2}.$
The operator $H(z)$ is the rank one projection $P_{1},$ so the above will be zero if the operator above applied to the constant function $f(z) = 1$ is zero. 
The operator $T(\psi)^{-1}H(\psi) $ is the same as the operator $T(\psi_{+}^{-1})H(\psi_{+})$ since it agrees with this operator on all trigonometric polynomials. Hence  $$T(\psi)^{-1}H(\psi) \,f = T(\psi_{+}^{-1})H(\psi_{+}) \,f=  (\psi_{+}^{-1}( \psi_{+} - (\psi_{+})_{0})/z) = (1 - \psi_{+}^{-1}(\psi_{+})_{0})/z.$$ We also have $H(\tilde{\psi})\,f = P(\tilde{\psi} /z).$ Finally  $$H(\tilde{\psi} /z)((1 - \psi_{+}^{-1}(\psi_{+})_{0})/z )= P(\tilde{\psi}/z) - (1-\psi_{+}^{-1}(\psi_{+})_{0}^{-1})/z.$$  Putting the three terms together yields the desired result.
\end{proof}

\begin{theorem}
Suppose that $\psi^{-1} = \tilde{\psi}$, that the operator $T(\psi)$ is invertible, and that $P_1T \iv(\psi) P_1=1$.
Then $ I - H(z\iv \psi)$ is also invertible
and its inverse is given by
\[ T(\psi^{-1})^{-1} (I-H(z\iv \psi^{-1})) T(\psi)^{-1}\,.\]
\end{theorem}
\begin{proof}
The inveribility of $T(\psi)$ implies the invertibility of $T(\psi\iv)$ (see the beginning of the proof of Theorem \ref{t5.1}).
This also means that the operator $T(\psi) T(\psi^{-1}) = I - H(\psi)^{2}$ is invertible.

We do the same steps as in the first two cases, except that this time we are dealing with the term
\[  
T(\psi) \, (I - H(\psi^{-1}/z) \,T(\psi)^{-1} = T(\psi) \, (I - H(\tilde{\psi}/z)) \,T(\psi)^{-1}\] in a slightly different manner.
First 
\[ T(\psi) \, (I - H(\tilde{\psi}/z)) \,T(\psi)^{-1} = I + H(\psi) T(z) + H(\psi) H(z) H(\tilde{\psi}) T(\psi)^{-1} .\]
Multiplying on the right by $I - H(\psi/z) $ we have that 
\[ T(\psi) \, (I - H(\tilde{\psi}/z)) \,T(\psi)^{-1} (I - H(\psi/z))  \]
is the same as 
\[ I  + H(\psi) H(z) H(\tilde{\psi}) T(\psi)^{-1} - H(\psi) T(z)\,H(\psi/z) - H(\psi) H(z) H(\tilde{\psi}) T(\psi)^{-1} H(\psi/z) .\]
The above uses the identity $H(\psi/z) = H(\psi)\,T(z) = T(1/z)\,H(\psi).$ Now $$H(\psi) \,T(z)\,H(\psi/z) = H(\psi) \,T(z)\,T(1/z)H(\psi) = 
H(\psi) \,Q_{1}\,H(\psi).$$ This means the above is the same as
\[ I   - H(\psi) H(\psi) + H(\psi) H(z)H(\psi)+ H(\psi) H(z) H(\tilde{\psi}) T(\psi)^{-1}-H(\psi) H(z) H(\tilde{\psi}) T(\psi)^{-1} H(\psi/z) .\]
Thus we will have our result if we can show that the last three terms of this operator sum to zero. 
To do this we consider the transpose of
\[H(\psi)+ H(\tilde{\psi}) T(\psi)^{-1} - H(\tilde{\psi}) T(\psi)^{-1} H(\psi/z),\]
that is,
\[H(\psi)+  T(\tilde{\psi})^{-1}H(\tilde{\psi}) - H(\psi/z)  T(\tilde{\psi})^{-1}H(\tilde{\psi}).\]
Using Lemma \ref{lemma.sum}, with $\psi$ replaced by $\wt{\psi}$ the statement is proved. 
\end{proof}

\begin{theorem}
Suppose that $\psi^{-1} = \tilde{\psi}$, that the operator $T(\psi)$ is invertible, and that $P_1T \iv(\psi) P_1=1$.
Then $ I + H( z\psi)$ is also invertible
and the inverse is given by
\[ (T(\psi^{-1})+H(z))^{-1} (I+H(z\psi^{-1})) (T(\psi)+H(z))^{-1}.\]
\end{theorem}
\begin{proof}
 The inverse of $(T(\psi)+H(z))^{-1}$ is the same as $T(\psi)^{-1} ( I - \frac{1}{2} H(z)T(\psi)^{-1}) $ which can be verified using basic linear algebra. 
 Now consider the product 
 $$ T(\psi)^{-1}( I - \frac{1}{2} H(z)T(\psi)^{-1})\,(I + H( z\psi)).$$
 Using the factorization of $\psi$ we see this is the same as the operator
 $$ T(\psi_{+}^{-1}) [ (I - \frac{1}{2}H(z) )(T(\wt{\psi_{+}}) + H(z \psi_{+})] .$$
 Note that  $T(\psi_{+}^{-1})T(\wt{\psi_{+}}) = T(\psi)^{-1}$ and that 
 $$T(\psi_{+}^{-1})H(z \psi_{+}) = T(\psi_{+}^{-1}) T(\wt{\psi_{+}})H(z \psi_{+}\wt{\psi_{+}}^{-1}) = T(\psi)^{-1} H( z \psi)$$ and thus is also a bounded operator. 
 Now multiplying by $I + H( z\wt{\psi})$ and doing a similar computation yields         
 $$  (   I + H( z\wt{\psi}))   T(\psi)^{-1}( I - \frac{1}{2} H(z)T(\psi)^{-1})\,(I + H( z\psi)) = T(\wt{\psi}) + H(z).$$ which gives the desired result.
\end{proof}

 A concise way to say both of the previous theorems is:
The inverse of $M(1, \phi_{0} \psi)$ is given by
\[ (T(\psi^{-1})+H(\phi_{0}))^{-1} M(1, \phi_{0}\psi^{-1}) (T(\psi)+H(\phi_{0}))^{-1}.\]

\end{document}